\documentclass[a4paper,12pt]{article}
\usepackage{amsmath,amsfonts,amssymb,amsthm}
\usepackage{graphics}
\usepackage{color}
\usepackage{xcolor}
\usepackage{epsfig}
\newcommand{\ignore}[1]{}
\newcommand{\av}{-\hspace{-2.4ex}\int}
\newcommand{\stackit}[2]{\genfrac{}{}{0pt}{1}{#1}{#2}}

\newtheorem{proposition}{Proposition}
\newtheorem{theorem}{Theorem}

\newtheorem{lemma}{Lemma}

\def\Xint#1{\mathchoice
{\XXint\displaystyle\textstyle{#1}}%
{\XXint\textstyle\scriptstyle{#1}}%
{\XXint\scriptstyle\scriptscriptstyle{#1}}%
{\XXint\scriptscriptstyle\scriptscriptstyle{#1}}%
\!\int}
\def\XXint#1#2#3{{\setbox0=\hbox{$#1{#2#3}{\int}$ }
\vcenter{\hbox{$#2#3$ }}\kern-.6\wd0}}

\def\dashint{\Xint-}

\newcommand{\supp}{\operatorname{supp}}

\begin{document}

\author{Julian Fischer and Felix Otto}
\title{Sublinear growth of the corrector in stochastic homogenization: Optimal stochastic estimates for slowly decaying correlations}
\maketitle

\begin{abstract}
We establish sublinear growth of correctors in the context of stochastic homogenization of linear elliptic PDEs. In case of weak decorrelation and ``essentially Gaussian'' coefficient fields, we obtain optimal (stretched exponential) stochastic moments for the minimal radius above which the corrector is sublinear. Our estimates also capture the quantitative sublinearity of the corrector (caused by the quantitative decorrelation on larger scales) correctly. The result is based on estimates on the Malliavin derivative for certain functionals which are basically averages of the gradient of the corrector, on concentration of measure, and on a mean value property for $a$-harmonic functions.
\end{abstract}

\section{Introduction}

In the present work, we are concerned with the stochastic homogenization of linear elliptic equations of the form
\begin{align}
\label{BasicProblem}
-\nabla \cdot a\nabla u = f.
\end{align}
In stochastic homogenization of elliptic PDEs, $a$ is typically a uniformly elliptic and bounded coefficient field, chosen at random according to some stationary and ergodic ensemble $\langle\cdot\rangle$. On large scales (and for slowly varying $f$), one may then approximate the solution $u$ to the equation \eqref{BasicProblem} by the solution $u_{hom}$ to the so-called \emph{effective equation}
\begin{align}
\label{EffectiveEquation}
-\nabla \cdot a_{hom}\nabla u_{hom} = f,
\end{align}
which is a constant-coefficient equation with the so-called effective coefficient $a_{hom}$. Mathematically, this homogenization effect is encoded in growth properties of the \emph{corrector} (cf.\ below for a definition of the corrector).

\medskip

The goal of the present paper is to provide a fairly simple
proof of {\it quantified} sublinear growth of the corrector under
very mild assumptions on the decorrelation of the coefficient field $a$ under the ensemble $\langle \cdot \rangle$.
We do this in the context of coefficient fields that are essentially Gaussian.
More precisely, we consider coefficient fields $a$ which are obtained from a Gaussian random field by pointwise application of some (nonlinear) mapping, the role of the nonlinear map being basically to enforce uniform ellipticity and boundedness of our coefficient field.

\medskip

The motivation for this result is the following:
Gloria, Neukamm, and the second author \cite{GloriaNeukammOtto} have shown that \emph{qualitatively} sublinear growth of the (extended) corrector $(\phi,\sigma)$ (cf.\ below for a definition) entails a large-scale intrinsic $C^{1,\alpha}$ regularity theory for $a$-harmonic functions.
In a subsequent work \cite{FischerOtto}, the two authors of the present paper have shown that \emph{slightly quantified} sublinear growth of the corrector even leads to a large-scale intrinsic $C^{k,\alpha}$ regularity theory for any $k\in \mathbb{N}$. Therefore, the results of the present work show that even in case of ensembles with very mild decorrelation, for almost every realization of the coefficient field, $a$-harmonic functions have arbitrary intrinsic smoothness properties on large scales. Furthermore, our results enable us to estimate the scale above which this happens -- a random quantity -- in a stochastically optimal way.
Indeed, the motivation for the present work was to establish such a (necessarily intrinsic) higher order regularity theory under the weakest possible assumptions on the decay of correlations.

\medskip

By an {\it intrinsic} regularity theory we mean that the regularity is measured
in terms of objects intrinsic to the Riemannian geometry defined by the coefficient
field $a$, like the dimension of the space of $a$-harmonic functions of a certain algebraic
growth rate, or like estimates on the H\"older modulus of the derivative of $a$-harmonic functions as
measured in terms of their distance to $a$-linear functions.
An {\it extrinsic} large-scale regularity theory for $a$-harmonic functions in case of
random coefficients was initiated on the level of a $C^{0,\alpha}$ 
in \cite{BenjaminiCopinKozmaYadin,MarahrensOtto} and pushed to $C^{1,0}$ in \cite{ArmstrongSmart},
which significantly extended qualitative arguments from the periodic case \cite{AvellanedaLin} to
quantitative arguments in the random case.
However, an extrinsic regularity theory is limited to $C^{1,0}$, as can be seen considering the harmonic coordinates: Taking higher order polynomials into account does not increase the local approximation order.

\medskip

After this motivation, we now return to the discussion of the history on bounds
on the corrector, as they depend on assumptions on the stationary ensemble of coefficient fields.
Almost-sure sublinearity (always meant in a spatially averaged sense) of the corrector $\phi$ 
under the mere assumption of ergodicity
was a key ingredient in the original work on stochastic homogenization by Kozlov \cite{Kozlov} and
by Papanicolaou \& Varadhan \cite{PapanicolaouVaradhan}. 
Almost-sure sublinearity of the {\it extended} corrector $(\phi,\sigma)$, as is needed
for the large-scale intrinsic $C^{1,\alpha}$-regularity theory, was established
in \cite{GloriaNeukammOtto} under mere ergodicity. 

\medskip

Yurinskii \cite{Yurinskii} was the first to quantify sublinear growth under
general mixing conditions, however only capturing suboptimal rates even in case of finite range of dependence. 
Very recently, a much improved quantification of sublinear growth of $\phi$ under finite
range assumptions was put forward by Armstrong, Kuusi, and Mourrat \cite{ArmstrongKuusiMourrat},
relying on a variational approach to quantitative stochastic homogenization introduced
by Armstrong and Smart \cite{ArmstrongSmart}, an approach which presumably can be extended to the
case of non-symmetric coefficients and more general mixing conditions following
\cite{ArmstrongMourrat}. However, while this approach gives the optimal, i.\ e.\ Gaussian,
stochastic integrability, it presently fails to give the optimal growth rates.

\medskip

Optimal growth rates have been obtained under a quantification of ergodicity different from finite
range or mixing conditions, namely under Spectral Gap assumptions on the ensemble.
This functional analytic tool from statistical mechanics was introduced into the field
of stochastic homogenization in an unpublished paper by Naddaf and Spencer \cite{NaddafSpencer},
and further leveraged by Conlon et.\ al.\ \cite{ConlonNaddaf,ConlonSpencer},
yielding optimal rates for some errors in stochastic homogenization in case of a small
ellipticity contrast. The work of Gloria, Neukamm and the second author extended these results
to the present case of arbitrary ellipticity contrast \cite{GloriaNeukammOttoInventiones,GloriaOtto1,GloriaOtto2}, in particular yielding
at most logarithmic growth of the corrector (and its stationarity in $d>2$).
Loosely speaking, the assumption of a Spectral Gap Inequality amounts to correlations
with integrable tails; in the above-mentioned works it has been used for discrete media (i.\ e.\ random
conductance models), but has subsequently been extended to the continuum case
\cite{GloriaOttoCorrectorEquation,GloriaOttoPeriodicApproximation}.

\medskip

A strengthening of the Spectral Gap Inequality is given by the Logarithmic Sobolev
Inequality (LSI); it is a slight strengthening in terms of the assumption (still essentially encoding
integrable tails of the correlations), but a substantial improvement in its effect,
since it implies Gaussian concentration of measure for Lipschitz random variables. 
The assumption of LSI and implicitly concentration of measure, which will be 
explicitly used in this work, has been
introduced into stochastic homogenization by Marahrens and the second author \cite{MarahrensOtto}.
In \cite{GloriaNeukammOtto}, it has been shown that the concept of LSI can be adapted to also capture
ensembles with slowly decaying correlations, i.\ e.\ thick non-integrable tails,
by adapting the norm of the vertical or Malliavin derivative to the correlation structure.
As a result, the stochastic integrability of the optimal rates could be improved from
algebraic to (stretched) exponential, but missing the expected Gaussian integrability.

\medskip

The main merit of the present contribution w.\ r.\ t.\ to \cite{GloriaNeukammOtto} is twofold: First, our approach directly provides optimal quantitative sublinearity of the corrector $(\phi,\sigma)$ on all scales above a random minimal radius $r_\ast$, i.e.\ in contrast to the estimates of \cite{GloriaNeukammOtto} our estimates capture the decorrelation on scales larger than $r_\ast$ in a single argument. Note that our definition of $r_\ast$ differs from the one in \cite{GloriaNeukammOtto}. Second, in case of weak decorrelation, our simpler arguments are nevertheless sufficient to establish optimal stochastic moments for the minimal radius $r_\ast$ above which the corrector $(\phi,\sigma)$ displays the quantified sublinear growth.

\medskip

In the present work, we consider the following type of ensembles on $\lambda$-uniformly elliptic tensor fields 
$a=a(x)$ on $\mathbb{R}^d$: 
Let $\tilde a=\tilde a(x)$ be a tensor-valued Gaussian random field on $\mathbb{R}^d$
that is centered (i.\ e.\ of vanishing expectation) and stationary (i.\ e.\ invariant under translation)
and thus characterized by the covariance $\langle \tilde a(x)\otimes \tilde a(0)\rangle$.
Our only additional assumption on $\tilde a$ is that there exists an exponent $\beta\in (0,d)$ such that
\begin{equation}\label{o1}
\big|\langle\tilde a(x)\otimes \tilde a(0)\rangle\big|\le |x|^{-\beta}\quad\mbox{for all}\;x\in\mathbb{R}^d.
\end{equation}
In this work, we are concerned with the case of weak decay of correlation in the sense of $\beta\ll 1$.
Let $\Phi$ be a 1-Lipschitz map from the space of tensors into the space of $\lambda$-elliptic symmetric tensors.
Then our ensemble is the distribution of $a$ where $a$ is given by $a(x):=\Phi(\tilde a(x))$.
Note that the normalization in the constant in (\ref{o1}) and in the Lipschitz constant is not essential,
since it can be achieved by a rescaling of $x$ and the amplitude of $\tilde a$.

\medskip

Concerning the mathematical tools of our approach, several ideas are inspired by the work \cite{GloriaNeukammOtto}. In particular, a key component of our approach are sensitivity estimates (Malliavin derivative bounds) for certain integral functionals, which basically average the gradient $\nabla (\phi,\sigma)$ over an appropriate cube. Furthermore, we rely on a mean-value property for $a$-harmonic functions, which has been derived in \cite{GloriaNeukammOtto} under appropriate smallness assumptions on the corrector. In our present contribution, we however pursue a conceptually simpler route to estimate the Malliavin derivative: The sensitivity estimate is performed through appropriate $L^q$-norm bounds and Meyer's estimate, rather than a more involved $\ell^2-L^1$-norm bound like in \cite{GloriaNeukammOtto}.

\medskip

Before stating our main results, let us recall the concept of correctors in homogenization and introduce some notation.
The basic idea underlying the concept of correctors in homogenization is the observation that the oscillations in the gradient $\nabla u_{hom}$ of solutions to the homogenized (constant-coefficient) problem \eqref{EffectiveEquation} occur on a much larger scale than the oscillations in the gradient $\nabla u$ of solutions to the original problem \eqref{BasicProblem}. Thus, it is important to understand how to add oscillations to an affine map (an affine map being always $a_{hom}$-harmonic) to obtain an $a$-harmonic map.
In the context of stochastic homogenization, one is therefore interested in constructing random scalar fields $\phi_i=\phi_i(a,x)$ subject to the equation
\begin{equation}\label{f4}
-\nabla\cdot a(e_i+\nabla\phi_i)=0
\end{equation}
which almost surely display sublinear growth in $x$: The $\phi_i$ then facilitate the transition from the $a_{hom}$-harmonic (Euclidean) coordinates $x\mapsto x_i$ to the ``$a$-harmonic coordinates'' $x\mapsto x_i+\phi_i(x)$. Since any affine map may be represented in the form $b+\sum_i \xi_i x_i$ for $b,\xi_i\in \mathbb{R}$, the $\phi_i$ also facilitate the construction of associated $a$-harmonic ``corrected affine maps'' $b+\sum_i \xi_i(x_i+\phi_i)$.

\medskip

With the help of the corrector, one may characterize the effective coefficient $a_{hom}$: In our setting of stochastic homogenization, the effective coefficient is given by the formula
\begin{align}
\label{EffectiveCoefficient}
a_{hom} e_i = \big\langle a(e_i+\nabla \phi_i) \big\rangle,
\end{align}
where $\langle \cdot \rangle$ refers to the expectation with respect to our ensemble (i.e.\ probability measure).

\medskip

In the language of a conducting medium with conductivity tensor $a$ -- note that in this picture, one has $f\equiv 0$ in \eqref{BasicProblem} -- , the quantity $E_i:=e_i+\nabla\phi_i$ corresponds to the (curl-free) ``microscopic'' electric field associated with a ``macroscopic'' electric field $e_i$ (and, therefore, $\phi_i$ corresponds to the ``microscopic'' correction to the ``macroscopic'' electric potential $x_i$). The corresponding (divergence-free) ``microscopic'' current density is given by
\begin{align}\label{DefinitionCurrent}
q_i :=a(e_i+\nabla\phi_i),
\end{align}
while the ``macroscopic'' current density associated with the ``macroscopic'' electric field $e_i$ is given by the ``average'' of this quantity, i.e.\ by the expression (\ref{EffectiveCoefficient}).

\medskip

In periodic homogenization of linear elliptic PDEs, it turns out to be convenient to introduce a dual quantity to the corrector $\phi_i$ (cf.\ e.g.\ \cite[p.27]{JikovKozlovOleinik}): One constructs a tensor field $\sigma_{ijk}$, skew-symmetric in the last two indices, which is a potential for the flux correction $q_i-a_{hom}e_i$ in the sense
\begin{equation}\label{f1}
\nabla\cdot\sigma_{i}=a(e_i+\nabla\phi_i)-a_{hom}e_i,
\end{equation}
where we have set $(\nabla\cdot\sigma_{i})_j:=\sum_{k=1}^d\partial_k\sigma_{ijk}$.
With the help of this ``extended corrector'' $(\phi, \sigma)$, it is possible to give a bound on the homogenization error (in terms of appropriate norms of $\phi$ and $\sigma$).

\medskip

One of the main merits of \cite{GloriaNeukammOtto} is the discovery of the usefulness of this extended corrector $(\phi,\sigma)$ in the context of stochastic homogenization. For stationary and ergodic ensembles
$\langle\cdot\rangle$ of $\lambda$-uniformly elliptic and symmetric coefficient fields $a=a(x)$ on $\mathbb{R}^d$, in \cite{GloriaNeukammOtto} correctors $\phi_i$ and $\sigma_{ijk}$ such that
\begin{equation}\label{f2}
\nabla\phi_i,\nabla\sigma_{ijk}\;
\begin{array}{l}
\mbox{are stationary,}\\ 
\mbox{of bounded second moment,}\\
\mbox{and of vanishing expectation,}
\end{array}
\end{equation}
have been constructed. As a consequence of this and of ergodicity, the $\phi_i$ and $\sigma_{ijk}$ almost surely display sublinear growth.
Note that in case of $\sigma_i$, the choice of the appropriate gauge is important for the property (\ref{f2}) and for our work, as the equation \eqref{f1} determines $\sigma_i$ (which by its skew-symmetry and its behavior under change of coordinates may be identified with a $d-1$-form) only up to the exterior derivative of a $d-2$-form. In fact, the choice of the gauge in \cite{GloriaNeukammOtto} is such that
\begin{equation}\label{f5}
-\triangle\sigma_{ijk}=\partial_jq_{ik}-\partial_kq_{ij},
\end{equation}
which in view of \eqref{f4} and \eqref{DefinitionCurrent} is clearly compatible with (\ref{f1}).


\medskip

\medskip

{\bf Notation.} 
To quantify the ellipticity and boundedness of our coefficient fields, throughout the paper we shall work with the assumptions
\begin{align}
av\cdot v&\geq \lambda |v|^2\quad\text{for all }v\in \mathbb{R}^d,
\label{aEllipticity}
\\
|av|&\leq |v|\quad\text{for all }v\in \mathbb{R}^d,
\label{aBoundedness}
\end{align}
where $\lambda\in (0,1)$. Note that in view of rescaling, the upper bound \eqref{aBoundedness} on $a$ does not induce a loss of generality of our results.

For our convenience, throughout the paper we shall assume that our coefficient field $a$ is symmetric. The arguments however easily carry over to the case of non-symmetric coefficient fields by simultaneously considering the correctors for the dual equation (i.e. the PDE with coefficient field $a^\ast$, $a^\ast$ denoting the transpose of $a$).

The expression $s\lesssim t$ is an abbreviation for $s\le C t$ with $C$ a generic constant only depending on the dimension $d$, the exponent $\beta>0$, and the ellipticity ratio $\lambda>0$.

The expression $s\ll t$ stands for $s\le \frac{1}{C} t$ with $C$ a generic sufficiently large constant only depending on the dimension $d$, the exponent $\beta>0$, and the ellipticity ratio $\lambda>0$.

By $I(E)$ we denote the characteristic function of an event $E$.

The notation $\dashint_A f$ refers to the average integral over the set $A$, i.e. we have $\dashint_A f \,dx=\int_A f \,dx/\int_A 1 \,dx$.

In the sequel, $(\phi,\sigma)$ stands for any component $\phi_i,\sigma_{ijk}$ for $i,j,k=1,\cdots,d$.

\section{Main Results and Structure of Proof}

Let us now state our main theorem.
To quantify the sublinear growth of the extended corrector $(\phi,\sigma)$, we first quantify the decay of spatial averages of $\nabla(\phi,\sigma)$ over larger scales.
In view of the decorrelation assumption (\ref{o1}) for our ensemble of coefficient fields, we expect that, up to logarithms, it is the exponent $\frac{\beta}{2}$ that governs the decay of averages of $\nabla (\phi,\sigma)$ and the improvement over linear growth for $(\phi,\sigma)$. Indeed, this exponent is reflected in the theorem.

\begin{theorem}\label{main}
Let $\tilde a=\tilde a(x)$ be a tensor-valued Gaussian random field on
$\mathbb{R}^d$ that is centered (i.\ e.\ of vanishing expectation)
and stationary (i.\ e.\ invariant under translation); assume that the covariance of $\tilde a$ satisfies the estimate
\begin{equation}\label{o1_b}
\big|\langle\tilde a(x)\otimes \tilde a(0)\rangle\big|\le |x|^{-\beta}\quad\mbox{for all}\;x\in\mathbb{R}^d
\end{equation}
for some $\beta\in (0,d)$. Let $\Phi:\mathbb{R}^{d\times d}\rightarrow \mathbb{R}^{d\times d}$ be a Lipschitz map with Lipschitz constant $\leq 1$; suppose that $\Phi$ takes values in the set of symmetric matrices subject to the ellipticity and boundedness assumptions \eqref{aEllipticity}, \eqref{aBoundedness}. Define the ensemble $\langle\cdot\rangle$ as the probability distribution of $a$, where $a$ is the image of $\tilde a$ under pointwise application of the map $\Phi$, i.e.\ $a(x):=\Phi(\tilde a(x))$.

\medskip

Assume in addition on the ensemble $\langle\cdot\rangle$ that $\beta$ in (\ref{o1_b}) is sufficiently 
small in the sense of
\begin{align}
\beta\le\frac{1}{C},
\label{AssumptionBetaSmallness}
\end{align}
where $C$ denotes a generic constant only depending on $d$ and $\lambda$.
\begin{itemize}
\item[i)] Consider a linear functional $F=Fh$ on vector fields $h=h(x)$ satisfying the boundedness property
\begin{equation}\label{o73}
|Fh|\le\big(\av_{|x|\le r}|h|^\frac{2d}{d+\beta}dx\big)^\frac{d+\beta}{2d}
\end{equation}
for some radius $r>0$. Then the random variable
$F\nabla(\phi,\sigma)$ satisfies uniform Gaussian bounds in the sense of
\begin{equation}\label{o72}
\langle I(|F\nabla(\phi,\sigma)|\ge M)\rangle\le C\exp(-\frac{1}{C}r^\beta M^2)\quad\mbox{for all}\;M\le 1.
\end{equation}
\item[ii)] There exists a (random) radius $r_*$ for which the ``iterated logarithmic'' bound
\begin{equation}\label{o62}
\frac{1}{r^2}\av_{|x|\le r}|(\phi,\sigma)-\av_{|x|\le r}(\phi,\sigma)|^2dx\le(\frac{r_*}{r})^\beta
\log(e+\log(\frac{r}{r_*}))\quad\mbox{for}\;r\ge r_*
\end{equation}
holds and which satisfies the stretched exponential bound
\begin{equation}\label{o17}
\langle\exp(\frac{1}{C}r_*^\beta)\rangle\le C.
\end{equation}
\end{itemize}
\end{theorem}

Morally speaking, Theorem \ref{main} converts statistical information on the coefficient field $a$
(or rather $\tilde a$)
into statistical information on the coefficient field $\nabla\phi:=\nabla(\phi_1,\cdots,\phi_d)$ related by
(\ref{f4}). Despite the nonlinearity of the map $a\mapsto\nabla\phi$, which only
in its linearization around $a={\rm id}$ turns into the Helmholtz projection, Theorem \ref{main}
states that $\nabla\phi$ essentially inherits the statistics of $a$: 
(\ref{o72}) implies in particular that spatial averages $F=\av_{|x|\le r}\nabla\phi dx$ of $\nabla\phi$
satisfy the same bounds as if $\nabla\phi$ itself was Gaussian with correlation decay (\ref{o1}).
On the level of these Gaussian bounds, the only prize to pay for nonlinearity is the restriction
$M\lesssim 1$ in (\ref{o72}) on the threshold.

\medskip

Incidentally, the way we obtain ii) from i) bears similarities with an argument in
\cite{ArmstrongKuusiMourrat} in the sense that a decomposition into Haar wavelets
is implicitly used.

\medskip

To obtain an estimate like \eqref{o72}, the starting point of our proof is the Gaussian concentration of measure applied to $\tilde a$.
%
\begin{proposition}[Concentration of Measure, cf.\ e.g.\ {\cite[Proposition 2.18]{Ledoux}})]
\label{ConcentrationOfMeasure}
Let $\tilde a=\tilde a(x)$ be a tensor-valued Gaussian random field on
$\mathbb{R}^d$ that is centered and stationary;
denote its covariance operator by $\operatorname{Cov}$.
Consider a random variable $F$, that is, a function(al) $F=F(\tilde a)$. Suppose
that $F$ is 1-Lipschitz in the sense that its functional derivative,
or rather its Fr\'echet derivative with respect to $L^2(\mathbb{R}^d;\mathbb{R}^{d\times d})$,
$\frac{\partial F}{\partial\tilde a}=\frac{\partial F}{\partial\tilde a}(\tilde a,x)$,
which can be considered a random tensor field and assimilated with a Malliavin derivative, satisfies
\begin{equation}\label{o2}
\int_{\mathbb{R}^d}\frac{\partial F}{\partial \tilde a}(\tilde a,x)
(\mbox{Cov}\frac{\partial F}{\partial \tilde a}(\tilde a,\cdot))(x)dx
\le 1\quad\mbox{for almost every}\;\tilde a.
\end{equation}

Then $F$ has Gaussian moments in the sense of
\begin{equation}\label{GaussianMoments}
\langle\exp(M (F-\langle F\rangle))\rangle\le\exp(\frac{M^2}{2})
\quad\mbox{for all}\;M\ge 0.
\end{equation}
Furthermore, for any $M\geq 0$ we have the estimate
\begin{equation}\label{o56}
\langle I(|F-\langle F\rangle|\ge M)\rangle\le 2\exp(-\frac{M^2}{2}).
\end{equation}
\end{proposition}

We now substitute our assumption (\ref{o2}) on the Fr\'echet derivative by a stronger but more tractable condition.
\begin{lemma}\label{SimplifiedConditionMalliavinDerivative}
Let $\tilde a=\tilde a(x)$ be a tensor-valued Gaussian random field on
$\mathbb{R}^d$ that is centered and stationary;
denote its covariance operator as $\operatorname{Cov}$ and suppose that for some $\beta\in (0,d)$ we have the bound
\begin{align*}
\big|\langle\tilde a(x)\otimes \tilde a(0)\rangle\big|\le |x|^{-\beta}\quad\mbox{for all}\;x\in\mathbb{R}^d.
\end{align*}
Let $\Phi:\mathbb{R}^{d\times d} \rightarrow \mathbb{R}^{d\times d}$ be a $1$-Lipschitz map; denote the probability distribution of $\Phi(\tilde a)$ as $\langle\cdot\rangle$.
Consider a functional $F$ on the space of tensor fields $\tilde a$ of the form $F=F(a)$ with $a(x):=\Phi(\tilde a(x))$; we shall use the abbreviation $F(\tilde a)$ for $F(\Phi(\tilde a))$. Let $q\in (1,2)$ be given by
\begin{align}\label{o12}
\frac{1}{q}=1-\frac{\beta}{2d}
\end{align}
and suppose that the Fr\'echet derivative of $F$ with respect to $L^2(\mathbb{R}^d;\mathbb{R}^{d\times d})$ satisfies
\begin{equation}\label{o46}
\Big(\int|\frac{\partial F}{\partial a}|^q dx\Big)^\frac{2}{q}\ll 1
\quad\mbox{for }\langle\cdot\rangle\text{-almost every }\;a.
\end{equation}
Then the estimate (\ref{o2}) is satisfied, i.e. we have
\begin{equation*}
\int_{\mathbb{R}^d}\frac{\partial F}{\partial \tilde a}(\tilde a,x)
(\mbox{Cov}\frac{\partial F}{\partial \tilde a}(\tilde a,\cdot))(x)dx
\leq 1\quad\mbox{for almost every}\;\tilde a.
\end{equation*}
\end{lemma}
We observe that if $q$ and $\beta$ are related by (\ref{o12}), as $\beta\uparrow d$ we have $q\uparrow 2$ and for $\beta\downarrow0$ we have $q\downarrow1$. 

\medskip

For linear functionals of (the gradient of) the corrector (which are therefore nonlinear functionals of the coefficient field $a$), we now establish an explicit representation of the Fr\'echet derivative; this will aid us in verifying the Lip\-schitz condition \eqref{o46} and thus ultimately the concentration of measure statements \eqref{GaussianMoments} and \eqref{o56} for (an appropriate modification of) such functionals.
\begin{lemma}\label{RepresentationMalliavin}
Consider a linear functional on $L^\frac{p}{p-1}(\mathbb{R}^d;\mathbb{R}^d)$ of the form
\begin{align}
\label{LinearFunctional}
Fh:=\int_{\mathbb{R}^d} g\cdot h\,dx,
\end{align}
where $g\in L^p(\mathbb{R}^d;\mathbb{R}^d)$, $p\geq 2$, and $\supp g\subset \{|x|\leq r\}$ for some $r\geq 1$.
Let $a$ be some coefficient field subject to the ellipticity and boundedness conditions (\ref{aEllipticity}), (\ref{aBoundedness}). Then the following two assertions hold:

\begin{itemize}
\item[1)] Consider the Fr\'echet derivative $\frac{\partial F}{\partial a}$ of the functional $F:=F\nabla \sigma_{ijk}$ (note that this functional is nonlinear in $a$, although it is linear in $\sigma_{ijk}$) at $a$ (for some fixed $i,j,k$). Introduce the decaying solutions $v$, $\tilde v_{jk}$ to the equations
\begin{equation}\label{o7}
-\triangle v=\nabla\cdot g
\end{equation}
and (where $a^\ast$ denotes the transpose of $a$)
\begin{equation}\label{o9}
-\nabla\cdot a^\ast \big(\nabla \tilde v_{jk}+(\partial_j ve_k-\partial_k ve_j)\big)=0.
\end{equation}
We then have the representation
\begin{equation}\label{o6}
\frac{\partial F}{\partial a}(a)=\big(\partial_jve_k-\partial_kve_j+\nabla\tilde v_{jk}\big)\otimes (\nabla\phi_i+e_i).
\end{equation}
\item[2)] Consider the Fr\'echet derivative $\frac{\partial F}{\partial a}$ of the functional $F:=F\nabla \phi_i$ at $a$. Introduce the decaying solution $\overline{v}$ to the equation (again, $a^\ast$ denoting the transpose of $a$)
\begin{equation}\label{oDefOverlineV}
-\nabla\cdot a^\ast \nabla \overline{v} = \nabla \cdot g.
\end{equation}
We then have the representation
\begin{align}
\label{RepresentationDelFPhi}
\frac{\partial F}{\partial a}(a) = \nabla \overline{v} \otimes (\nabla \phi_i + e_i).
\end{align}
\end{itemize}
\end{lemma}

The previous explicit representation of the Fr\'echet derivative for certain linear functionals of (the gradient of) the corrector $(\phi,\sigma)$ enables us to verify the bound (\ref{o46}) for the Malliavin derivative, provided that a certain mean value property is satisfied for $a$-harmonic functions. Note that the latter requirement is a condition on the coefficient field $a$; in Lemma \ref{SufficientConditionMeanValueProperty} below we shall provide a sufficient condition for this property to hold.
\medskip

As the functionals which the next lemma shall be applied to are basically averages of $\nabla \phi$ or $\nabla \sigma$ over cubes of a certain scale $r$, we state the lemma in a form which makes it directly applicable in such a setting. In particular, the boundedness assumption (\ref{o90}) for the linear functional is motivated by these considerations.

\begin{lemma}\label{SmallnessMalliavinDerivative}
Let $r\geq 1$ and consider a linear functional $h\mapsto Fh$ on $L^\frac{p}{p-1}(\mathbb{R}^d;\mathbb{R}^d)$ (with $p\in (2,\infty)$) satisfying the support and boundedness condition
\begin{equation}\label{o90}
|Fh|\le
\big(\dashint_{|x|\le r}|h|^\frac{p}{p-1}dx\big)^\frac{p-1}{p}.
\end{equation}
Suppose that the constraint
\begin{align}
\label{SmallnessBetaThroughp}
2<p<2+c(d,\lambda)
\end{align}
holds (with $c(d,\lambda)>0$ to be fixed in the proof below).
Let $q\in (1,2)$ be related to $p$ through
\begin{equation}\label{o13}
\frac{1}{p}=\frac{1}{q}-\frac{1}{2}.
\end{equation}
Consider the Fr\'echet derivative $\frac{\partial F}{\partial a}$ of the functional $F:=F\nabla \sigma_{ijk}$ (or the functional $F:=F\nabla \phi_i$; note that these functionals are nonlinear functionals of $a$) at some symmetric coefficient field $a$ subject to the conditions (\ref{aEllipticity}), (\ref{aBoundedness}).

Provided that the coefficient field $a$ is such that the mean value property
\begin{align}\label{o14}
\dashint_{|x|\leq \rho} |\nabla u|^2 dx
\lesssim \dashint_{|x|\leq R} |\nabla u|^2 dx
\quad\text{for any }R\geq r\text{ and any }\rho\in [r,R]
\end{align}
holds for any $a$-harmonic function $u$ and provided that furthermore $a$ is such that
\begin{align}\label{SublinearityCorrector}
\lim_{R\rightarrow \infty} \frac{1}{R}\left(\dashint_{|x|\leq R} |(\phi,\sigma)-\dashint_{|x|\leq R} (\phi,\sigma)|^2 ~dx\right)^{1/2}=0
\end{align}
is satisfied, we have the estimate
\begin{align}
\label{o15}
\Big(\int|\frac{\partial F}{\partial a}|^qdx\Big)^\frac{2}{q}
\lesssim r^{-\frac{(p-2)d}{p}}.
\end{align}
\end{lemma}
Note that for $q$ related to $\beta$ through (\ref{o12}) and $p$ related to $q$ through (\ref{o13}), we have $r^{-\frac{(p-2)d}{p}}=r^{-\beta}$, i.e. by \eqref{o15} the $L^q$-norm of the Malliavin derivative decays like $r^{-\frac{\beta}{2}}$. This demonstrates that for functionals like our averages of $\nabla (\phi,\sigma)$ -- note that these functionals have vanishing expectation due to the vanishing expectation of $\nabla (\phi,\sigma)$ -- , the concentration of measure indeed improves on large scales with the desired exponent: The ``typical value'' of the average of $\nabla (\phi,\sigma)$ on some scale $r$ decays like $r^{-\frac{\beta}{2}}$.

\medskip

We now have to provide a sufficient condition for the mean value property for $a$-harmonic functions (\ref{o14}). To do so, we make use of the following result from \cite{GloriaNeukammOtto}, which provides the mean-value property assuming just an appropriate sublinearity condition on the corrector $(\phi,\sigma)$. 
\begin{proposition}[see {\cite[Lemma 2]{GloriaNeukammOtto}}]
\label{MeanValueForSmallCorrector}
There exists a constant $C_0$ only depending on dimension $d$ and ellipticity ratio $\lambda>0$
with the following property: Suppose that for an elliptic coefficient field $a$ subject to the ellipticity and boundedness conditions (\ref{aEllipticity}) and (\ref{aBoundedness}) the scalar
and vector potentials $(\phi,\sigma)$, cf.\ (\ref{f4}) and (\ref{f1}), satisfy
\begin{equation}\label{o40}
\Big(\av_{|x|\le R}|(\phi,\sigma)-\av_{|x|\le R}(\phi,\sigma)|^2dx\Big)^\frac{1}{2}\le \frac{1}{C_0}R
\quad\mbox{for all}\;R\ge r.
\end{equation}
Then for any two radii $R\ge r$ and $\rho\in [r,R]$ and any $a$-harmonic function $u$ in $\{|x|\le R\}$ we have
\begin{equation}\nonumber
\av_{|x|\le \rho}|\nabla u|^2dx\lesssim \av_{|x|\le R}|\nabla u|^2dx.
\end{equation}
\end{proposition}
We shall show in the proof of the next lemma that the quantitative sublinearity condition on the corrector (\ref{o40}) may be reduced
to a smallness assumption on a certain family of linear functionals of the gradient of the corrector. Basically, these functionals will be obtained by averaging the gradient of the corrector on appropriate cubes, cf.\ the proof below. In combination with the previous proposition, we get the following lemma.

Note that this result will allow us to buckle, since by Lemma \ref{SmallnessMalliavinDerivative} and Lemma~\ref{SimplifiedConditionMalliavinDerivative}
the mean-value property (\ref{o14}) and thus ultimately (\ref{o45}) implies
\begin{equation}\nonumber
\Big(\int|\frac{\partial F_{n,R}}{\partial a}|^qdx\Big)^\frac{2}{q}\lesssim R^{-\beta}.
\end{equation}
\begin{lemma}\label{SufficientConditionMeanValueProperty}
There exist $N\lesssim 1$ linear functionals $F_1,\cdots,F_N$ satisfying
the support and boundedness condition
\begin{align*}
|F_n h|\lesssim \big(\int_{|x|\leq 1} |h|^\frac{p}{p-1} ~dx \big)^\frac{p-1}{p}
\end{align*}
such that the following holds:
Denote by $F_{n,R}$ the rescaling of $F_n$ given by $F_{n,R}h=F_n\big(h(\frac{\cdot}{R})\big)$. Let $r\geq 0$.
Provided that the condition
\begin{equation}\label{o45}
\sup_{R\ge r\;\text{dyadic};n=1,\cdots,N}F_{n,R}\nabla(\phi,\sigma)\ll 1
\end{equation}
and the condition \eqref{SublinearityCorrector}
%
are satisfied, we have the smallness estimate \eqref{o40} for the corrector;
in particular, by Proposition \ref{MeanValueForSmallCorrector} the mean value property (\ref{o14}) holds for any $a$-harmonic function $u$ on scales $\geq r$, i.e. we have
\begin{align*}
\dashint_{|x|\leq \rho} |\nabla u|^2 dx
\lesssim \dashint_{|x|\leq R} |\nabla u|^2 dx
\quad\text{for any }R\geq r\text{ and any }\rho\in [r,R].
\end{align*}
\end{lemma}

With these preparations, we are able to establish our main theorem.
The main technical difficulty in the proof below is that our estimate
\begin{align*}
\Big(\int \big|\frac{\partial F}{\partial a}\big|^q dx \Big)^\frac{2}{q}
\lesssim r^{-\beta}
\end{align*}
for the Malliavin derivative of linear functionals of the gradient of the corrector (cf. (\ref{o15})) is a conditional bound: It relies on the assumption that the mean-value property (\ref{o14}) holds for $a$-harmonic functions on scales larger than $r$. For the concentration of measure estimate (\ref{GaussianMoments}), however, an unconditional estimate of the form (\ref{o2}) or (\ref{o46}) (the latter being a proxy for (\ref{o2})) is needed. By Lemma \ref{SufficientConditionMeanValueProperty} we know that the mean-value property holds, provided that for a certain family of linear functionals of the corrector the smallness estimate
\begin{align*}
\sup_{R\geq r~\text{dyadic}; n=1,\cdots,N} F_{n,R} \nabla(\phi,\sigma) \leq \frac{1}{C_0}
\end{align*}
is satisfied ($C_0$ being a universal constant). To circumvent this problem, in the proof below we therefore introduce the family of functionals
\begin{align*}
\bar F_r := \min\Big\{\sup_{R\geq r~\text{dyadic}; n=1,\cdots,N} F_{n,R} \nabla(\phi,\sigma),\frac{1}{C_0}\Big\},
\end{align*}
for which by design the unconditional bound for the Malliavin derivative
\begin{align*}
\Big(\int \big|\frac{\partial \bar F_r}{\partial a}\big|^q dx \Big)^\frac{2}{q}
\lesssim r^{-\beta}
\end{align*}
holds. Therefore, concentration of measure is applicable to $\bar F_r$. The remainder of the proof of the first part of our theorem below is dedicated to handling the (a priori unknown) expectation $\langle\bar F_r\rangle$.

\medskip The proof of the second assertion of our main theorem will mainly rely on the first assertion of the theorem as well as the quantitative improvement of the Malliavin derivative of averages of $(\nabla \phi,\nabla \sigma)$ on larger scales, as captured by the estimate (\ref{o15}).

\section{Concentration of Measure and Estimates of the Malliavin Derivative}

\subsection{Concentration of Measure}

\begin{proof}[Proof of Proposition \ref{ConcentrationOfMeasure}]
For the proof of the concentration of measure estimate (\ref{GaussianMoments}), we refer the reader to \cite[Proposition 2.18]{Ledoux}.
We now establish (\ref{o56}).
By Chebychev's inequality, (\ref{GaussianMoments}) implies $\langle I(F-\langle F\rangle\ge M)\rangle \le\exp(-\frac{M^2}{2})$. In combination with the same estimate with $F$ replaced by $-F$, we obtain (\ref{o56}).
\end{proof}

\begin{proof}[Proof of Lemma \ref{SimplifiedConditionMalliavinDerivative}]
We need to verify that the condition (\ref{o2}) is implied by the assumption (\ref{o46}).

To do so, we first note that by H\"older's inequality we have for any exponent $1<q<\infty$
\begin{eqnarray*}
\int\frac{\partial F}{\partial \tilde a}
\mbox{Cov}\frac{\partial F}{\partial \tilde a}dx
&\le&\Big(\int|\frac{\partial F}{\partial \tilde a}|^qdx\Big)^\frac{1}{q}
\Big(\int|\mbox{Cov}\frac{\partial F}{\partial \tilde a}|^\frac{q}{q-1}dx\Big)^\frac{q-1}{q}.
\end{eqnarray*}
Since $\mbox{Cov}$ is the convolution with $\langle\tilde a(x)\tilde a(0)\rangle$ and since we have the bound $|\langle\tilde a(x)\otimes \tilde a(0)\rangle|\leq |x|^{-\beta}$,
we have for the second factor
\begin{eqnarray*}
\Big(\int|\mbox{Cov}\frac{\partial F}{\partial \tilde a}|^\frac{q}{q-1}dx\Big)^\frac{q-1}{q}
&\le&\Big(\int\Big|\int\frac{1}{|x-y|^\beta}
|\frac{\partial F}{\partial \tilde a}(y)|dy\Big|^\frac{q}{q-1}dx\Big)^\frac{q-1}{q},
\end{eqnarray*}
which allows us to use the Hardy-Littlewood-Sobolev inequality
\begin{eqnarray*}
\Big(\int\Big|\int\frac{1}{|x-y|^\beta}
|\frac{\partial F}{\partial \tilde a}(y)|dy\Big|^\frac{q}{q-1}dx\Big)^\frac{q-1}{q}
&\lesssim&\Big(\int|\frac{\partial F}{\partial \tilde a}|^qdx\Big)^\frac{1}{q},
\end{eqnarray*}
provided the exponents $q$ and $\beta$ are related by (\ref{o12}).
From this string of inequalities we learn that (\ref{o2}) also holds provided
\begin{equation}\label{o4}
\Big(\int|\frac{\partial F}{\partial\tilde a}|^qdx\Big)^\frac{2}{q}\ll 1
\quad\mbox{for almost every}\;\tilde a.
\end{equation}
We now change variables according to $a(x)=\Phi(\tilde a(x))$; by the chain rule for $F(\tilde a)=F(\Phi(\tilde a))$
we have
$\frac{\partial F}{\partial\tilde a}(\tilde a,x)=\Phi'(\tilde a(x))\frac{\partial F}{\partial a}(a,x)$,
so that by the 1-Lipschitz continuity of $\Phi$, our assumption (\ref{o46})
implies (\ref{o4}) and thus (\ref{o2}).
\end{proof}

\subsection{Representation of the Malliavin derivative}

\begin{proof}[Proof of Lemma \ref{RepresentationMalliavin}]
We first give the argument for the ``vector potential'' $\sigma$, fixing a component $\sigma_{ijk}$.
Consider a functional of the form $F:=F\nabla \sigma_{ijk}$ with $Fh$ as in (\ref{LinearFunctional}). We claim that the Fr\'echet derivative of $F$ with respect to $a$ is given by (\ref{o6}) where the functions $v=v(x)$ and $\tilde v_{jk}=\tilde v_{jk}(a,x)$ are determined as the decaying solutions of the elliptic equations (\ref{o7}) and (\ref{o9}).
%

Computing the functional derivative of $F$ as a function of $a$ amounts
to a linearization. We thus consider an arbitrary tensor field $\delta a=\delta a(x)$, 
which we think of as an infinitesimal perturbation of $a$, and which thus generates infinitesimal perturbations
$\delta\phi$ and $\delta\sigma$ of $\phi$ and $\sigma$ according to (\ref{f4}), (\ref{DefinitionCurrent}), and (\ref{f5}), that is,
\begin{equation}\label{o10}
-\nabla\cdot(a\nabla\delta\phi_i+\delta a(\nabla\phi_i+e_i))=0
\end{equation}
and
\begin{equation}\label{o8}
-\triangle\delta\sigma_{ijk}=\partial_j\big(\delta a(\nabla\phi_i+e_i)+a\nabla\delta\phi_i\big)_k
-\partial_k\big(\delta a(\nabla\phi_i+e_i)+a\nabla\delta\phi_i\big)_j.
\end{equation}
In terms of the infinitesimal perturbation $\delta F$ of $F$, this implies by integration by parts
(or rather by directly appealing to the weak Lax-Milgram formulations of the elliptic equations)
\begin{eqnarray*}
\delta F
&=&\int g\cdot \nabla\delta\sigma_{ijk} dx\\
&\stackrel{(\ref{o7})}{=}&-\int \nabla v \cdot \nabla \delta\sigma_{ijk} dx\\
&\stackrel{(\ref{o8})}{=}&\int(\partial_jve_k-\partial_kve_j)\cdot
\big(\delta a(\nabla\phi_i+e_i)+a\nabla\delta\phi_i\big) dx
\\
&\stackrel{(\ref{o9})}{=}&\int(\partial_jve_k-\partial_kve_j)\cdot\delta a(\nabla\phi_i+e_i)dx
-\int\nabla\tilde v_{jk}\cdot a\nabla\delta\phi_i dx\\
&\stackrel{(\ref{o10})}{=}&\int\big(\partial_jve_k-\partial_kve_j+\nabla\tilde v_{jk}\big)\cdot\delta a(\nabla\phi_i+e_i)dx,
\end{eqnarray*}
which is nothing else than (\ref{o6}).

\medskip

Let us now establish the second part of our lemma. Consider a functional of the scalar potential of the form $F:=F\nabla \phi_i$.
To represent its Fr\'echet derivative, introduce the decaying solution $\overline{v}$ to the equation (\ref{oDefOverlineV}).
We observe that the variation of $F$ with respect to $a$ is given by
\begin{eqnarray*}
\delta F &=& \int g \cdot \nabla \delta \phi_i dx
\\
&\stackrel{(\ref{oDefOverlineV})}{=}&-\int a^\ast\nabla \overline{v}
\cdot \nabla \delta \phi_i dx
\\
&\stackrel{(\ref{o10})}{=}&\int \nabla \overline{v} \cdot
\delta a(\nabla \phi_i+e_i) dx,
\end{eqnarray*}
which leads to the conclusion (\ref{RepresentationDelFPhi}).
\end{proof}


\subsection{Sensitivity estimate}

\begin{proof}[Proof of Lemma \ref{SmallnessMalliavinDerivative}]
We now argue that under certain boundedness assumptions on $F=Fh$
as a linear functional in vector fields $h=h(x)$,
we control the size (\ref{o46}) of its Fr\'echet derivative 
$\frac{\partial F}{\partial a}=\frac{\partial F}{\partial a}(a,x)$ 
as a nonlinear functional $F\nabla\sigma_{ijk}=F(a)$
in coefficient fields $a=a(x)$ (and similarly in the case $F(a)=F\nabla\phi_i$; for this case, the (simpler) proof is sketched afterwards).

To this aim, let us first note that we have a Calderon-Zygmund estimate
for $-\nabla\cdot a\nabla$ with the exponents $p$ and its dual exponent $\frac{p}{p-1}$:
For any decaying function $w$ and vector field $h$ on $\mathbb{R}^d$ related by
\begin{equation}\label{o22}
-\nabla\cdot a\nabla w=\nabla\cdot h
\end{equation}
we have
\begin{equation}\label{o11}
\int|\nabla w|^\frac{p}{p-1}dx\lesssim \int|h|^\frac{p}{p-1}dx\quad\mbox{and}\quad
\int|\nabla w|^pdx\lesssim \int|h|^pdx.
\end{equation}
This assertion holds by Meyer's estimate (see e.g.\ \cite{Meyers}), which only requires the ellipticity and boundedness assumptions \eqref{aEllipticity}, \eqref{aBoundedness} on $a$ as well as the estimate $|p-2|\ll 1$, which is ensured by our condition \eqref{SmallnessBetaThroughp}. Note that an analogous estimate would hold for the dual equation $-\nabla \cdot a^\ast \nabla w=\nabla \cdot h$ if our coefficient field were nonsymmetric.

\medskip

In the following, we will use the abbreviation $\|\cdot\|_{p,B}$ for
the spatial $L^p$-norm on the set $B$; we write $\|\cdot\|_{p}$ when $B=\mathbb{R}^d$.
We start by arguing that because $\frac{p}{p-1}\in(1,2)$, (\ref{o14}) also entails
\begin{equation}\label{o31}
\av_{|x|\le \rho}|\nabla u|^\frac{p}{p-1}dx\lesssim \av_{|x|\le R}|\nabla u|^\frac{p}{p-1}dx.
\end{equation}
It is obviously enough to establish (\ref{o31}) only for $R\ge 2\rho$; hence by Jensen's
inequality, (\ref{o31}) follows from (\ref{o14}) once we establish the reverse H\"older
inequality
\begin{equation}\label{o33}
\big(\av_{|x|\le\frac{R}{2}}|\nabla u|^2dx\big)^\frac{1}{2}\lesssim\
\av_{|x|\le R}|\nabla u|dx.
\end{equation}
%
To this purpose, we test
$-\nabla\cdot a\nabla u=0$ with $\eta^{2\gamma}(u-m)$, where $\eta$ is a smooth cut-off of $\chi_{\{|x|\le \frac{R}{2}\}}$ in
$\{|x|\le R\}$ (with the property $|\nabla \eta|\lesssim \frac{1}{R}$) and where the exponent $\gamma\geq 1$ and the constant $m\in\mathbb{R}$ will be chosen later. By the ellipticity and boundedness assumptions \eqref{aEllipticity}, \eqref{aBoundedness} and Young's inequality we obtain
\begin{equation}\nonumber
\int(\eta^\gamma|\nabla u|)^2dx\lesssim \int((u-m)|\nabla\eta^\gamma|)^2dx,
\end{equation}
and thus
\begin{equation}\nonumber
\int|\nabla(\eta^\gamma(u-m))|^2dx\lesssim\int((u-m)|\nabla\eta^\gamma|)^2dx,
\end{equation}
which by the estimate on $\nabla\eta$ gives
\begin{equation}\label{o32}
\|\nabla(\eta^\gamma(u-m))\|_2\lesssim \frac{1}{R} \|\eta^{\gamma-1}(u-m)\|_2.
\end{equation}
%
%
%
On the r.\ h.\ s.\ of (\ref{o32}) we use first H\"older's inequality, then the isoperimetric inequality on $\{|x|\le R\}$ and finally Sobolev's inequality on the whole space (for simplicity, we assume $d>2$ here)
\begin{eqnarray*}\nonumber
\|\eta^{\gamma-1}(u-m)\|_2&\le&\|\eta^{\gamma}(u-m)\|_{\frac{2d}{d-2}}^\frac{\gamma-1}{\gamma}
\|u-m\|_{\frac{d}{d-1},|x|\le R}^\frac{1}{\gamma}\\
&\lesssim&\|\nabla(\eta^{\gamma}(u-m))\|_{2}^\frac{\gamma-1}{\gamma}\|\nabla u\|_{1,|x|\le R}^\frac{1}{\gamma},
\end{eqnarray*}
provided the exponent $\gamma\in (1,\infty)$ is chosen such that $\frac{1}{2}=(1-\frac{1}{\gamma})\frac{d-2}{2d}+\frac{1}{\gamma}\frac{d-1}{d}$
(which -- as a simple computation shows -- is satisfied precisely for $\gamma=\frac{d}{2}$)
and the constant $m$ is the spatial average of $u$ on $\{|x|\le R\}$. The combination of the last two estimates yields
\begin{align*}
\|\nabla(\eta^\gamma(u-m))\|_{2}\lesssim \frac{1}{R}
\|\nabla(\eta^{\gamma}(u-m))\|_{2}^\frac{\gamma-1}{\gamma}
\|\nabla u\|_{1,|x|\le R}^\frac{1}{\gamma},
\end{align*}
which (by $\gamma=\frac{d}{2}$) entails $\|\nabla u\|_{2,|x|\le \frac{R}{2}}\le\|\nabla(\eta^\gamma(u-m))\|_{2}\lesssim R^{-\frac{d}{2}} \|\nabla u\|_{1,|x|\le R}$ and thus (\ref{o33}).

\medskip

We now give the argument for (\ref{o15}) in case of a functional of the form $F\nabla \sigma_{ijk}$ (the case $F\nabla \phi_i$ will be treated below). Clearly (\ref{o90}) implies that there exists
a (deterministic) vector field $g=g(x)$ with
\begin{equation}\label{o19}
\supp g\subset \{|x|\le r\}\quad\mbox{and}\quad \|g\|_p\le r^{-\frac{p-1}{p}d}
\end{equation}
such that we have the representation for $F=Fh$ as a linear functional on vector fields $h=h(x)$
\begin{equation}
\label{RepresentFwithg}
Fh=\int g\cdot h\,dx.
\end{equation}
This gives us access to the representation (\ref{o6})
of its Fr\'echet derivative $\frac{\partial F}{\partial a}$
considered as a nonlinear functional $F\nabla\sigma_{ijk}=F(a)$ of $a$. 
Using this representation, a partition
into dyadic annuli, and H\"older's estimate (recall (\ref{o13})) we obtain
\begin{align}
&\|\frac{\partial F}{\partial a}\|_{q}
\nonumber\\
&\lesssim
\|(|\nabla v|+|\nabla\tilde v_{jk}|)|\nabla\phi_i+e_i|\|_{q,|x|\le 2r}
\nonumber
\\
&~~~
+\sum_{n=1}^\infty\|(|\nabla v|+|\nabla\tilde v_{jk}|)|\nabla\phi_i+e_i|\|_{q,2^n r\le |x|\le 2^{n+1} r}
\nonumber
\\
&\lesssim
(\|\nabla v\|_{p,|x|\le 2r}+\|\nabla\tilde v_{jk}\|_{p,|x|\le 2r})\|\nabla\phi_i+e_i\|_{2,|x|\le 2r}
\nonumber
\\
&~~~
+\sum_{n=1}^\infty(\|\nabla v\|_{p,2^n r\le|x|\le2^{n+1} r}+\|\nabla\tilde v_{jk}\|_{p,2^n r\le|x|\le2^{n+1} r})
\|\nabla\phi_i+e_i\|_{2,2^n r\le |x|\le 2^{n+1} r}.\label{o16}
\end{align}
In view of (\ref{o14}) applied to the $a$-harmonic function $u(x)=x_i+\phi_i(x)$, cf.\ (\ref{f4}), we obtain for all radii $\rho \ge r$ using Caccioppoli's inequality and (\ref{SublinearityCorrector})
%
\begin{align*}
&\int_{|x|\le \rho}|\nabla\phi_i+e_i|^2dx\lesssim \rho^d\liminf_{R\uparrow\infty}\av_{|x|\le R}|\nabla\phi_i+e_i|^2dx
\\&
\lesssim \rho^d \liminf_{R\uparrow\infty} \frac{1}{R^2} \dashint_{|x|\leq 2R} \Big|\phi_i+x_i-\dashint_{|x|\leq 2R} (\phi_i+x_i) \Big|^2 dx
\lesssim \rho^d.
\end{align*}
Hence (\ref{o16}) turns into
\begin{eqnarray}
\|\frac{\partial F}{\partial a}\|_{q}
&\lesssim& r^\frac{d}{2}(\|\nabla v\|_{p}+\|\nabla\tilde v_{jk}\|_{p})\label{o18}\\
&&+\sum_{n=1}^\infty (2^{n} r)^\frac{d}{2}(\|\nabla v\|_{p,|x|\ge2^{n} r}
+\|\nabla\tilde v_{jk}\|_{p,|x|\ge2^{n} r}).\label{o20}
\end{eqnarray}
It thus remains to estimate the auxiliary functions $v$ and $\tilde v_{jk}$. 
The estimate of the terms in line (\ref{o18}) is easy: By (\ref{o19})
and Calderon-Zygmund for (\ref{o7}) we
obtain $\|\nabla v\|_{p}\lesssim \|g\|_{p}\le r^{-\frac{p-1}{p}d}$. 
By (\ref{o11}) we have Calderon-Zygmund
with exponent $p$ for the equation (\ref{o9}), so that 
$\|\nabla\tilde v_{jk}\|_{p}\lesssim\|\nabla v\|_{p}\lesssim r^{-\frac{p-1}{p}d}$.
In order to control the terms in line (\ref{o20}), we shall establish the following estimates
for $n\in\mathbb{N}$
\begin{eqnarray}
\|\nabla v\|_{p,|x|\ge 2^n r}&\lesssim&
(2^n)^{-d+\frac{d}{p}}
r^{-\frac{p-1}{p}d}
,\label{o21}\\
\|\nabla \tilde v_{jk}\|_{p,|x|\ge 2^n r}&\lesssim&
n(2^n)^{-d+\frac{d}{p}}
r^{-\frac{p-1}{p}d}.\label{o23}
\end{eqnarray}
We note that since $p>2$, these estimates imply that the sum over $n$ in (\ref{o20}) converges and gives (\ref{o15}).

\medskip

The estimate (\ref{o21}) for the solution $v$ of the constant coefficient equation (\ref{o7})
is classical: We already argued that $\|\nabla v\|_{p}\lesssim r^{-\frac{p-1}{p}d}$; by the estimate on the support of $g$ in (\ref{o19}) we have that $v$ is harmonic in $\{|x|\ge r\}$ and that
it has vanishing flux $\int_{|x|=r}x\cdot \nabla v=0$. It thus decays as
$|\nabla v(x)|\lesssim |x|^{-d} r^{d-\frac{d}{p}}
\|\nabla v\|_{p}$ for $|x|\ge 2r$,
which in particular yields (\ref{o21}). We now turn to (\ref{o23}) and to this purpose rewrite the equation
(\ref{o9}) for $\tilde v_{jk}$ as
\begin{equation}\nonumber
-\nabla\cdot a^\ast \nabla\tilde v_{jk}=\nabla\cdot\tilde g
\end{equation}
with the r.\ h.\ s.\ $\tilde g:=-a^\ast (\partial_jve_k-\partial_kve_j)$. We already argued that 
$\|\tilde g\|_{p}\lesssim r^{-\frac{p-1}{p}d}$ and (\ref{o21}) translates into
\begin{equation}\label{o29}
\|\tilde g\|_{p,|x|\ge 2^n r}\lesssim(2^n)^{-d+\frac{d}{p}} r^{-\frac{p-1}{p}d}.
\end{equation}
In order to proceed, we split $\tilde g$ into $\{\tilde g_m\}_{m=0,1,\cdots}$ where
$\tilde g_0$ is supported in $\{|x|\le 2r\}$ and for $m\ge 1$ $\tilde g_m$ is supported
in $\{2^m r\le|x|\le 2^{m+1} r\}$, so that (\ref{o29}) translates into
\begin{equation}\label{o25}
\|\tilde g_m\|_{p}\lesssim(2^m)^{-d+\frac{d}{p}}r^{-\frac{p-1}{p}d}.
\end{equation}
This entails a splitting of $\tilde v_{jk}$ into $\{\tilde v_m\}_{m=0,1,\cdots}$,
where $\tilde v_m$ is the Lax-Milgram solution of
\begin{equation}\label{o26}
-\nabla\cdot a^\ast \nabla\tilde v_m=\nabla\cdot\tilde g_m.
\end{equation}
We will now argue that
\begin{equation}\label{o27}
\|\nabla\tilde v_m\|_{p,|x|\ge 2^{n}r}\lesssim \min\{(2^n)^{-d+\frac{d}{p}},(2^m)^{-d+\frac{d}{p}}\}r^{-\frac{p-1}{p}d},
\end{equation}
which implies the estimate (\ref{o23}) by the triangle inequality 
$\|\nabla\tilde v_{jk}\|_{p,|x|\ge 2^{n} r}\le\sum_{m=0}^\infty\|\nabla\tilde v_m\|_{p,|x|\ge 2^{n} r}$.
We note that (\ref{o25}) together with our Calderon-Zygmund estimate (\ref{o11})
applied to (\ref{o26}) yields $\|\nabla\tilde v_m\|_{p}\lesssim(2^m)^{-d+\frac{d}{p}}r^{-\frac{p-1}{p}d}$.
In order to establish (\ref{o27}), it thus remains to show
\begin{equation}\label{o30}
\|\nabla\tilde v_m\|_{p,|x|\ge 2^{n}r}\lesssim (2^n)^{-d+\frac{d}{p}}r^{-\frac{p-1}{p}d}\quad\mbox{for}\;m<n.
\end{equation}
We argue in favor of (\ref{o30}) by duality and thus consider an arbitrary $h\in L^\frac{p}{p-1}$ supported in 
$\{|x|\ge 2^n r\}$ and denote by $w$ the corresponding Lax-Milgram solution of (\ref{o22}).
By integration by parts,
we deduce from
(\ref{o22}) and (\ref{o26}) that $\int h\cdot\nabla\tilde v_m dx=\int \tilde g_m\cdot\nabla w\,dx$.
By the support condition on $\tilde g_m$ this yields
\begin{equation}\nonumber
\Big|\int h\cdot\nabla\tilde v_m dx\Big|\le\|\tilde g_m\|_{p}\|\nabla w\|_{\frac{p}{p-1},|x|\le 2^{m+1}r}.
\end{equation}
By the support assumption on $h$ we have that $w$ is $a$-harmonic in $\{|x|\le 2^n r\}$.
Since $m<n$, we may use (\ref{o31}) applied to $w$ in form of
\begin{equation}\nonumber
\|\nabla w\|_{\frac{p}{p-1},|x|\le 2^{m+1}r}\lesssim (2^{n-m})^{-d+\frac{d}{p}}\|\nabla w\|_{\frac{p}{p-1},|x|\le 2^{n}r}.
\end{equation}
We combine this with (\ref{o11}) in form of $\|\nabla w\|_{\frac{p}{p-1}}\lesssim\|h\|_{\frac{p}{p-1}}$,
and with (\ref{o25}), to obtain
\begin{equation}\nonumber
\Big|\int h\cdot\nabla\tilde v_m dx\Big|\lesssim (2^{n})^{-d+\frac{d}{p}} r^{-\frac{p-1}{p}d}\|h\|_{\frac{p}{p-1}},
\end{equation}
which gives (\ref{o30}).

\medskip

In the case of a functional of the scalar potential of the form $F(a)=F\nabla \phi_i$, we claim that the Fr\'echet derivative of $F$ is again controlled in the sense of (\ref{o15}).
The proof is mostly analogous to the previous one; we again rewrite $F$ as in (\ref{RepresentFwithg}) with some $g$ satisfying (\ref{o19}). Starting from the representation (\ref{RepresentationDelFPhi}), one derives an analogue of estimate (\ref{o16}) reading
\begin{eqnarray*}
\lefteqn{\|\frac{\partial F}{\partial a}\|_{q}}
\\
&\lesssim&\|\nabla \overline{v}\|_{p,|x|\le 2r}\|\nabla\phi_i+e_i\|_{2,|x|\le 2r}
\\
&&+\sum_{n=1}^\infty\|\nabla \overline{v}\|_{p,2^n r\le|x|\le2^{n+1} r}
\|\nabla\phi_i+e_i\|_{2,2^n r\le |x|\le 2^{n+1} r}.
\end{eqnarray*}
The second factors on the right in this estimate coincide with the ones in the case $F(a)=F\nabla \sigma_{ijk}$; therefore, we get the following analogue to estimate (\ref{o20}):
\begin{eqnarray*}
\|\frac{\partial F}{\partial a}\|_{q}
&\lesssim& r^\frac{d}{2}\|\nabla \overline{v}\|_{p}
+\sum_{n=1}^\infty (2^{n}r)^\frac{d}{2}\|\nabla \overline{v}\|_{p,|x|\ge 2^{n} r}.
\end{eqnarray*}
The equation (\ref{oDefOverlineV}) for $\overline{v}$ has the structure of the equation (\ref{o26}) with $m=0$ (including the estimate $||g||_{p} \lesssim r^{-\frac{p-1}{p}d}$ and the inclusion $\operatorname{supp} g\subset \{|x|\leq 2r\}$, cf. \eqref{o19}); therefore, the decay property (\ref{o30}) carries over to our $\overline{v}$. This establishes the estimate
\begin{align*}
\|\frac{\partial F}{\partial a}\|_{q} \lesssim r^{-\frac{p-2}{2p}d}
\end{align*}
also in the case $F(a)=F\nabla \phi_i$.
\end{proof}

\medskip

\subsection{Sufficient conditions for the mean value property in terms of linear functionals of the corrector}

%
%
%
%
%
%

\medskip

\begin{proof}[Proof of Lemma \ref{SufficientConditionMeanValueProperty}]
In order to show that (\ref{o45}) and (\ref{SublinearityCorrector}) imply (\ref{o14}), we only need to show the existence of functionals $F_1,\cdots,F_N$ such that (\ref{o45}) and (\ref{SublinearityCorrector}) imply
\begin{equation}\label{o46bis}
\frac{1}{R}\Big(\av_{|x|\le R}|(\phi,\sigma)-\av_{|x|\le R}(\phi,\sigma)|^2 dx\Big)^\frac{1}{2}\ll 1
\quad\mbox{for all}\;R\ge r.
\end{equation}
By Proposition \ref{MeanValueForSmallCorrector}, the estimate (\ref{o14}) follows from (\ref{o46bis}).

\medskip

Let us now give the argument for (\ref{o46bis}). First, it is clearly enough to show  that for any $0<\delta\ll 1$, there exists $N\lesssim\delta^{-d}$ functionals $F_1,\cdots,F_N$ on vector fields which are bounded in the sense of
\begin{equation}\label{o93}
|F_nh|\lesssim\delta^{-d}\big(\int_{|x|\le 1}|h|^\frac{p}{p-1}dx\big)^\frac{p-1}{p}
\end{equation}
and such that for any dyadic $\rho\ge 1$
\begin{equation}\label{o92}
\frac{1}{\rho}\Big(\av_{|x|\le \rho}\big|(\phi,\sigma)-\av_{|x|\le \rho}(\phi,\sigma)\big|^2 dx\Big)^\frac{1}{2}
\lesssim\delta+\sup_{R\ge \rho\;\text{dyadic}}\max_{n=1,\cdots,N}|F_{n,R}\nabla(\phi,\sigma)|.
\end{equation}
By dyadic iteration, it is enough to show for any dyadic $\rho\geq 1$
\begin{align}
&\frac{1}{\rho}\Big(\dashint_{|x|\le \rho}\big|(\phi,\sigma)-\av_{|x|\le \rho}(\phi,\sigma)\big|^2 dx\Big)^\frac{1}{2}\nonumber
\\ \nonumber
&\lesssim \max_{n=1,\cdots,N}|F_{n,\rho}\nabla(\phi,\sigma)|
\\&~~~~\label{o43}
+\delta\Big(\frac{1}{2\rho}\Big(\dashint_{|x|\le 2\rho}\big|(\phi,\sigma)-\av_{|x|\le 2\rho}(\phi,\sigma)\big|^2 dx\Big)^\frac{1}{2}+1\Big).
\end{align}
Indeed, abbreviating 
$D_m:=\frac{1}{2^{m}}\Big(\dashint_{|x|\le 2^m}\big|(\phi,\sigma)-\dashint_{|x|\le 2^m}(\phi,\sigma)\big|^2dx\Big)^\frac{1}{2}$,
the estimate (\ref{o43}) may be rewritten as (using a slight readjustment of $\delta$)
\begin{eqnarray}\nonumber
D_m
&\le&C_0\max_{n=1,\cdots,N}|F_{n,2^m}\nabla(\phi,\sigma)|
+\delta\Big(D_{m+1}+1\Big),
\end{eqnarray}
which may be iterated to 
\begin{align*}
D_m \le& \frac{1}{1-\delta}C_0\max_{M=m,m+1,\cdots,m+m_0}\max_{n=1,\cdots,N}|F_{n,2^{M}}\nabla(\phi,\sigma)|
\\&
+\frac{\delta}{1-\delta}+\delta^{m_0+1} D_{m+m_0+1}.
\end{align*}
By our sublinearity assumption on the corrector (\ref{SublinearityCorrector}) (which may be rewritten as $\lim_{m_0\uparrow\infty} D_{m_0}=0$), this yields (\ref{o92}).
\medskip

We now turn to the argument for \eqref{o43}.
By Caccioppoli's estimate on (\ref{f4}) we have
\begin{equation}\nonumber
\Big(\int_{|x|\le\frac{3}{2}\rho}|\nabla\phi_i|^2dx\Big)^\frac{1}{2}\lesssim
\frac{1}{\rho}
\Big(\int_{|x|\le 2\rho}\big(\phi_i-\av_{|x|\le 2\rho}\phi_i\big)^2 dx\Big)^\frac{1}{2}+\rho^{d/2},
\end{equation}
and thus in particular for the flux $q_i=a(\nabla\phi_i+e_i)$ 
\begin{equation}\nonumber
\Big(\int_{|x|\le\frac{3}{2}\rho}|q_i|^2dx\Big)^\frac{1}{2}\lesssim
\frac{1}{\rho}
\Big(\int_{|x|\le 2\rho}\big(\phi_i-\av_{|x|\le 2\rho}\phi_i\big)^2dx\Big)^\frac{1}{2}+\rho^{d/2}.
\end{equation} 
Caccioppoli's estimate on (\ref{f5}) gives
\begin{eqnarray*}
\lefteqn{\Big(\int_{|x|\le \rho}|\nabla\sigma_{ijk}|^2dx\Big)^\frac{1}{2}}
\nonumber\\
&\lesssim&
\frac{1}{\rho}
\Big(\int_{|x|\le \frac{3}{2}\rho}\big(\sigma_{ijk}-\av_{|x|\le \frac{3}{2}\rho}\sigma_{ijk}\big)^2 dx\Big)^\frac{1}{2}+
\Big(\int_{|x|\le \frac{3}{2}\rho}|q_i|^2dx\Big)^\frac{1}{2}\nonumber\\
&\le&
\frac{1}{\rho}
\Big(\int_{|x|\le 2\rho}\big(\sigma_{ijk}-\av_{|x|\le 2\rho}\sigma_{ijk}\big)^2dx\Big)^\frac{1}{2}+
\Big(\int_{|x|\le \frac{3}{2}\rho}|q_i|^2dx\Big)^\frac{1}{2}
\end{eqnarray*}
The last three estimates combine to
\begin{align}
&\nonumber
\Big(\int_{|x|\le \rho}|\nabla(\phi,\sigma)|^2 dx\Big)^\frac{1}{2}
\\&
\label{o89}
\lesssim\frac{1}{\rho}\Big(\int_{|x|\le 2\rho}\big|(\phi,\sigma)-\av_{|x|\le 2\rho}(\phi,\sigma)\big|^2 dx\Big)^\frac{1}{2}+\rho^{d/2}.
\end{align}
Hence for (\ref{o43}) is enough to show
\begin{eqnarray*}\nonumber
\frac{1}{\rho} \lefteqn{\Big(\dashint_{|x|\le \rho}\big|(\phi,\sigma)-\av_{|x|\le \rho}(\phi,\sigma)\big|^2 dx\Big)^\frac{1}{2}}\nonumber\\
&\le&\max_{n=1,\cdots,N}|F_{n,\rho}\nabla(\phi,\sigma)|
+\delta \Big(\dashint_{|x|\le \rho}|\nabla(\phi,\sigma)|^2dx\Big)^\frac{1}{2}.
\end{eqnarray*}
This statement is not just true for $(\phi,\sigma)-\av_{|x|\le \rho}(\phi,\sigma)$, but for any function $\zeta$ of vanishing spatial average on $\{|x|\leq \rho\}$: By rescaling, it is sufficient to show the estimate on the unit ball $\{|x|\leq 1\}$. It is more convenient to see it when the unit ball $\{|x|\le 1\}$ is replaced by the unit square $(0,1)^d$:
\begin{eqnarray}\label{o44}
\int_{(0,1)^d}\zeta^2dx
&\le&\max_{n=1,\cdots,N}|F_{n}\nabla\zeta|^2+\delta^2\int_{(0,1)^d}|\nabla\zeta|^2dx.
\end{eqnarray}
Indeed, dividing $(0,1)^d$ into $N=\delta^{-d}$ (suppose that $\delta^{-1}$ is an integer)
sub-cubes $\{Q_n\}_{n=1,\cdots,N}$ of side length $\delta$ and setting $F_n\nabla\zeta:=\av_{Q_n}\zeta dx$ (recall
that $\int_{(0,1)^d}\zeta dx=0$ so that $F_n$ is indeed a function of $\nabla\zeta$), 
(\ref{o44}) follows from using Poincar\'e's estimate on each $Q_n$
in form of $\int_{Q_n}\zeta^2 dx-|Q_n|(\av_{Q_n}\zeta dx)^2\lesssim\delta^2\int_{Q_n}|\nabla\zeta|^2dx$
and then summing up. We note that by Poincar\'e's estimate on $(0,1)^d$, the $F_n$ have the desired 
boundedness property (\ref{o93}), at first on gradient fields $\nabla\zeta$
\begin{equation}\nonumber
|F_n\nabla\zeta|\le\delta^{-d}\big(\int_{(0,1)^d}|\zeta|^\frac{p}{p-1} dx\big)^\frac{p-1}{p}
\lesssim \delta^{-d}\big(\int_{(0,1)^d}|\nabla\zeta|^\frac{p}{p-1} dx\big)^\frac{p-1}{p},
\end{equation}
and then on any vector field $h$ by extension \`a la Hahn-Banach.
\end{proof}

\section{Proof of Main Result}

\begin{proof}[Proof of Theorem \ref{main}]
~
\newline
{\bf Proof of Assertion i)}.


Consider the functionals $\{F_{n}\}_{n=1,\cdots,N}$ and their rescalings
$\{F_{n,R}\}_{n=1,\cdots,N;R\;\text{dyadic}}$ constructed in Lemma~\ref{SufficientConditionMeanValueProperty}. Let us abbreviate $F_{n,R}\nabla (\phi,\sigma)$ as $F_{n,R}$. We would like to apply concentration of measure to these functionals.

The main difficulty that we need to overcome is that our sensitivity estimate
(\ref{o15}) in Lemma \ref{SmallnessMalliavinDerivative} for the quantity $F_{n,r}$ is based on the assumption that the mean-value property (\ref{o14}) holds for $a$-harmonic functions down to scale $r$. By Lemma \ref{SufficientConditionMeanValueProperty} this assumption may be reduced to the smallness assumption (\ref{o45}) for our functionals $F_{m,R}$ on scales $R\geq r$, so that Lemma \ref{SmallnessMalliavinDerivative} becomes applicable under the assumption \eqref{o45}: 
Let $q$ be related to $\beta$ through (\ref{o12}) and let $p$ be related to $q$ through (\ref{o13}). By the smallness assumption on $\beta$ in our theorem (cf.\ \eqref{AssumptionBetaSmallness}), we deduce that (\ref{SmallnessBetaThroughp}) holds. By scaling, our functionals $F_{n,r}$ satisfy the estimate (\ref{o90}) up to a universal constant factor. Furthermore, by ergodicity the property (\ref{SublinearityCorrector}) holds for $\langle\cdot\rangle$-almost every coefficient field $a$ (regarding $\sigma$, this result has been shown in \cite[Lemma 1]{GloriaNeukammOtto}; for $\phi$, it is classical but may also be found in \cite{GloriaNeukammOtto}).
Thus, the estimate \eqref{o15} holds for $F_{n,r}$  under the assumption \eqref{o45}, i.e. there exists a constant $C_0$ only depending on $d$, $\lambda$, and $\beta$, such that for any $n=1,\cdots,N$ and any radius $r$ the implication
\begin{equation}\label{o55}
\sup_{m,R\ge r\text{ dyadic}}F_{m,R}\le\frac{1}{C_0}\quad\Longrightarrow\quad
\|\frac{\partial F_{n,r}}{\partial a}\|_q^2\lesssim\frac{1}{r^\beta}
\end{equation}
holds for $\langle\cdot\rangle$-a.e. coefficient field $a$.

\medskip

To apply concentration of measure in the form of Proposition \ref{ConcentrationOfMeasure} to some functional $F$, we however need an unconditional bound on the Malliavin derivative (cf.\ (\ref{o2}) respectively (\ref{o46})).

Therefore we first introduce a new random variable whose derivative vanishes whenever the smallness condition in (\ref{o55}) is violated: Consider the auxiliary random variable
\begin{equation}\label{o58}
\bar F_r:=\min\{\sup_{m,R\ge r}|F_{m,R}|,\frac{1}{C_0}\},
\end{equation}
where the $\sup$ runs over all dyadic radii $R=2^k r$, $k\in \mathbb{N}_0$.
By the usual differentiation rules applied to the Fr\'echet derivative $\frac{\partial}{\partial a}$
in the norm $\|\cdot\|_q$, we obtain
\begin{equation}\nonumber
\|\frac{\partial\bar F_r}{\partial a}\|_q\le I(\sup_{m,R\ge r}F_{m,R}\le\frac{1}{C_0})\sup_{m,R\ge r}
\|\frac{\partial F_{m,R}}{\partial a}\|_q
\end{equation}
and thus by (\ref{o55})
\begin{equation}\nonumber
\|\frac{\partial\bar F_r}{\partial a}\|_q^2\lesssim \frac{1}{r^\beta}.
\end{equation}
By Lemma \ref{SimplifiedConditionMalliavinDerivative}, we may apply concentration of measure in form of (\ref{o56}) to the random variable $c r^{\beta/2} \bar F_r$ (where $c$ is some small universal constant). This yields
\begin{equation}\label{o57}
\langle I(|\bar F_r-\langle\bar F_r\rangle|\ge M)\rangle\lesssim\exp(-\frac{1}{C}r^\beta M^2)
\quad\mbox{for all}\; M\geq 0,
\end{equation}
so that it remains to control the expectation $\langle\bar F_r\rangle$.

\medskip

Because of (\ref{f2}) and the definition of $F_{m,R}$, it follows from qualitative ergodicity
of $\langle\cdot\rangle$ and Birkhoff's ergodic theorem that $\lim_{R\uparrow\infty}F_{m,R}=0$
almost surely, so that by dominated convergence $\lim_{r\uparrow\infty}\langle\bar F_r\rangle=0$.
Hence there exists a {\it finite} radius $r_0$ which is minimal with the property 
\begin{equation}\label{o61}
\langle \bar F_r\rangle\le \frac{1}{4C_0}\quad\mbox{for all}\;r\ge r_0.
\end{equation}
Hence using (\ref{o57}) for $M=\frac{1}{4C_0}$ we get in view of the definition (\ref{o58}) of $\bar F_r$
\begin{equation}\label{o59}
\langle I(\sup_{m,R\ge r}|F_{m,R}|\ge\frac{1}{2C_0})\rangle
\lesssim\exp(-\frac{1}{C}r^\beta)\quad\mbox{for all}\;r\ge r_0.
\end{equation}

\medskip

On the basis of (\ref{o59}), we now get a {\it quantitative} estimate on $r_0$.
To this purpose we now consider the auxiliary variable
\begin{equation}\label{o60}
\bar F_{n,r}:=\eta(\sup_{m,R\ge r}F_{m,R})F_{n,r},
\end{equation}
where again the $\sup$ runs over all dyadic radii $R=2^k r$, $k\in \mathbb{N}_0$, and where the cut-off function $\eta=\eta(F)$ is given by
\begin{equation}\label{o51}
\eta(F)=\max\{\min\{2C_0(\frac{1}{C_0}-F),1\},0\}.
\end{equation}
The advantage of the auxiliary variable (\ref{o60}) over (\ref{o58}) is that we control
its expectation: Since the stationary $\nabla(\phi,\sigma)$ has vanishing expectation, cf.\ (\ref{f2}), 
and by the linearity of $F_{n,r}$ in $\nabla(\phi,\sigma)$ we have
$\langle F_{n,r}\rangle=0$ and thus $\langle\bar F_{n,r}\rangle=\langle(\eta-1) F_{n,r}\rangle$
so that by construction of $\eta$
\begin{equation}\nonumber
|\langle\bar F_{n,r}\rangle|\le\langle I(\sup_{m,R\ge r}|F_{m,R}|\ge\frac{1}{2C_0})|F_{n,r}|\rangle.
\end{equation}
Since the stationary $\nabla(\phi,\sigma)$ has bounded second moments, cf.\ (\ref{f2}), and by
the boundedness property of $F_{n,r}$ in $\nabla(\phi,\sigma)$ we obtain from the Cauchy-Schwarz inequality
\begin{equation}\nonumber
\langle\bar F_{n,r}\rangle^2\lesssim\langle I(\sup_{m,R\ge r}|F_{m,R}|\ge\frac{1}{2C_0})\rangle,
\end{equation}
which in view of (\ref{o59}) improves to
\begin{equation}\label{f54}
|\langle\bar F_{n,r}\rangle|\lesssim\exp(-\frac{1}{C}r^\beta)\quad\mbox{for any}\;r\ge r_0.
\end{equation}
By differentiation rules for the Fr\'echet derivative $\frac{\partial}{\partial a}$
in the norm $\|\cdot\|_q$ we obtain for the auxiliary
random variable $\bar F_{n,r}$
\begin{eqnarray*}
\|\frac{\partial \bar F_{n,r}}{\partial a}\|_q&\le&
I(\sup_{m,R\ge r}|F_{m,R}|\le \frac{1}{C_0})\big(2C_0|F_{n,r}|
\sup_{m,R\ge r}\|\frac{\partial F_{m,R}}{\partial a}\|_q+
\|\frac{\partial F_{n,r}}{\partial a}\|_q\big)\\
&\le&3I(\sup_{m,R\ge r}|F_{m,R}|\le \frac{1}{C_0})
\sup_{m,R\ge r}\|\frac{\partial F_{m,R}}{\partial a}\|_q,
\end{eqnarray*}
and thus by (\ref{o55})
\begin{equation}\nonumber
\|\frac{\partial \bar F_{n,r}}{\partial a}\|_q^2\lesssim\frac{1}{r^\beta},
\end{equation}
and hence by concentration of measure in form of (\ref{o56}) (applied to $cr^{\beta/2} \bar F_{n,r}$ by means of Lemma \ref{SimplifiedConditionMalliavinDerivative}, $c$ being a small universal constant)
\begin{equation}\nonumber
\langle I(|\bar F_{n,r}-\langle \bar F_{n,r}\rangle|\ge M)\rangle
\lesssim\exp(-\frac{1}{C}r^\beta M^2)\quad\mbox{for all}\;M\geq 0.
\end{equation}
Together with (\ref{f54}) this yields
\begin{equation}\nonumber
\langle I(|\bar F_{n,r}|\ge M)\rangle
\lesssim\exp(-\frac{1}{C}r^\beta M^2)\quad\mbox{for}\;M\gg\exp(-\frac{1}{C}r^\beta)\;\mbox{and}\;r\ge r_0.
\end{equation}
By definition (\ref{o60}) we have $\langle I(|F_{n,r}|\ge M)\rangle
%
%
\le\langle I(\sup_{m,R\ge r}|F_{m,R}|\ge\frac{1}{2C_0})\rangle+\langle I(|\bar F_{n,r}|\ge M)\rangle$
so that by (\ref{o59}) the above upgrades to
\begin{equation}\nonumber
\langle I(|F_{n,r}|\ge M)\rangle
\lesssim\exp(-\frac{1}{C}r^\beta M^2)
\quad\mbox{for}\;1\ge\;M\gg\exp(-\frac{1}{C}r^\beta)\;\mbox{and}\;r\ge r_0.
\end{equation}
Since $r^\beta\exp(-\frac{1}{C}r^\beta)\lesssim 1$ for all $r$, the above
holds without the lower restriction on $M$:
\begin{equation}\label{o70}
\langle I(|F_{n,r}|\ge M)\rangle
\lesssim\exp(-\frac{1}{C}r^\beta M^2)
\quad\mbox{for}\;M\le1\;\mbox{and}\;r\ge r_0.
\end{equation}
Using this estimate with $r$ replaced by $R$ and summing over the finite index
set $n=1,\cdots,N$ and all dyadic $R\ge r$ we obtain
\begin{eqnarray}\label{o63}
\langle I(\sup_{m,R\ge r}|F_{m,R}|\ge M)\rangle
&\lesssim&\exp(-\frac{1}{C}r^\beta M^2)\nonumber\\
&&\mbox{for}\;M\le1\;\mbox{and}\;r\ge r_0
\end{eqnarray}
and thus in particular for the auxiliary random variable (\ref{o58})
\begin{equation}\nonumber
\langle I(\bar F_r\ge M)\rangle
\lesssim\exp(-\frac{1}{C}r^\beta M^2)
\quad\mbox{for all}\;M\geq 0\;\mbox{and}\;r\ge r_0,
\end{equation}
where the upper bound on $M$ is immaterial since $\bar F_r\le\frac{1}{C_0}\le 1$.
Using $\langle \bar F_r\rangle=\int_0^\infty\langle I(\bar F_r\ge M)\rangle dM$, this yields
the following quantification of $\lim_{r\uparrow\infty}\langle\bar F_r\rangle=0$:
\begin{equation}\nonumber
\langle \bar F_r\rangle\lesssim\int_0^\infty\exp(-r^\beta M^2)dM
\lesssim r^{-\frac{\beta}{2}}
\quad\text{ for all }r\geq r_0.
\end{equation}
Since $r_0$ was minimal in (\ref{o61}) and since $\langle\bar F_r\rangle$ depends continuously on $r$, this yields the desired
\begin{equation}\label{o71}
r_0\lesssim 1.
\end{equation}

\medskip

It remains to argue why (\ref{o70}), which together with (\ref{o71}) may be
rephrased as
\begin{equation}\nonumber
\langle I(|F_{n,r}|\ge M)\rangle
\lesssim\exp(-\frac{1}{C}r^\beta M^2)
\quad\mbox{for}\;M\le1\;\mbox{and}\;r\gg 1,
\end{equation}
yields (\ref{o72}). It just suffices to include the given functional $F$ from (\ref{o73}) into the
list of finitely many functionals $F_1,\cdots,F_N$, say, as the last functional $F_N=F$,
and then to specify the above to $n=N$. We note that for $q$ related to $\beta$ through (\ref{o12}) and $p$ related to $q$ through (\ref{o13}) one has $\frac{p}{p-1}=\frac{2d}{d+\beta}$, i.e. (\ref{o73}) entails (\ref{o90}). Note that by adjusting the constants, (\ref{o72}) is
trivial for $r\lesssim 1$, so that we obtain (\ref{o72}) over the whole range $r\ge 0$.

\medskip


{\bf Proof of Assertion ii)}.

The arguments in this section require $\beta<2$, which in view
of our assumption $\beta\ll 1$ is no restriction. Let $r_*$ denote the minimal dyadic radius with the property (\ref{o62})
-- we know but don't have to use that it is finite by quantitative ergodicity. In order to establish
(\ref{o17}), it is enough to show for a given dyadic $r_0\ge 1$ that 
\begin{equation}\label{o77}
\langle I(r_*> r_0)\rangle\lesssim\exp(-\frac{1}{C}r_0^\beta).
\end{equation}
It will be convenient to replace balls by cubes. Moreover, all radii or rather side length are dyadic. 
By definition of $r_*$ as the smallest radius with (\ref{o62}), the event $r_*> r_0$ means that
there exists a radius $R\ge r_0$ with
\begin{align}
\label{o80}
&\frac{1}{R^2}\av_{(-R,R)^d}|(\phi,\sigma)-\av_{(-R,R)^d}(\phi,\sigma)|^2dx>(\frac{r_0}{R})^\beta f(\frac{R}{r_0}),
\end{align}
where
\begin{align*}
f(z):=\log(e+\log z).
\end{align*}
In the sequel, the intermediate (dyadic) radius $r_1\in[r_0,R]$ with
\begin{equation}\label{o83}
r_1\sim r_0^\frac{\beta}{2}R^{1-\frac{\beta}{2}}f^\frac{1}{2}(\frac{R}{r_0})\quad\mbox{so that}\quad
(\frac{r_1}{R})^2\sim(\frac{r_0}{R})^\beta f(\frac{R}{r_0})
\end{equation}
will play a role. Note that we use here $\beta>0$ and that $f(z)$ grows sub-algebraically.
For the l.\ h.\ s.\ of (\ref{o80}) we note
\begin{align}
\lefteqn{\av_{(-R,R)^d}|(\phi,\sigma)-\av_{(-R,R)^d}(\phi,\sigma)|^2dx}
\nonumber
\\
\label{o81}
&=\av_{(-R,R)^d}|(\phi,\sigma)-(\phi,\sigma)_{r_1}|^2dx
+\sum_{\stackit{r\in[2r_1,R]}{r\;\text{dyadic}}}\av_{(-R,R)^d}|(\phi,\sigma)_{\frac{r}{2}}-(\phi,\sigma)_{r}|^2dx,
\end{align}
where $(\phi,\sigma)_r$ denotes the $L^2((-R,R)^d)$-orthogonal projection of $(\phi,\sigma)$ onto
the space of functions that are piecewise constant on the $(\frac{R}{r})^d$ dyadic sub-cubes $Q$ of ``level $r$''
(that is, of side length $2r$)
of the cube $(-R,R)^d$. In other words, on such a sub-cube $Q$, $(\phi,\sigma)_r=\av_{Q}(\phi,\sigma)dx$.
With this language, we may rewrite the first r.\ h.\ s.\ term of (\ref{o81}) as
\begin{equation}\nonumber
\av_{(-R,R)^d}|(\phi,\sigma)-(\phi,\sigma)_{r_1}|^2dx
=(\frac{r_1}{R})^d\sum_{Q\;\text{level}\;r_1}\av_{Q}|(\phi,\sigma)-\av_{Q}(\phi,\sigma)|^2dx,
\end{equation}
so that by Poincar\'e's estimate on each of the cubes $Q$ we obtain
\begin{equation}\nonumber
\av_{(-R,R)^d}|(\phi,\sigma)-(\phi,\sigma)_{r_1}|^2dx
\lesssim r_1^2\av_{(-R,R)^d}|\nabla(\phi,\sigma)|^2dx,
\end{equation}
and then by Caccioppoli's estimate based on (\ref{f4}) \& (\ref{f5}), cf.\ (\ref{o89}),
\begin{align*}
&\av_{(-R,R)^d}|(\phi,\sigma)-(\phi,\sigma)_{r_1}|^2dx
\\&
\lesssim (\frac{r_1}{R})^2\av_{(-2R,2R)^d}|(\phi,\sigma)-\av_{(-2R,2R)^d}(\phi,\sigma)|^2dx+r_1^2.
\end{align*}
As a consequence of Lemma \ref{SufficientConditionMeanValueProperty}, there exist $N\sim 1$ linear functionals $\{F_{n}\}_{n=1,\cdots,N}$ whose rescaled versions $F_{n,r}$ satisfy the boundedness property (\ref{o73}) such that for any $r\ge 2R$ we have the implication
\begin{align*}
&\sup_{r\ge 2R~\text{dyadic}}\max_{n=1,\cdots,N}(F_{n,r}\nabla(\phi,\sigma))^2\ll 1
\\&
\Longrightarrow
\frac{1}{R^2}\av_{(-2R,2R)^d}|(\phi,\sigma)-\av_{(-2R,2R)^d}(\phi,\sigma)|^2dx\lesssim 1.
\end{align*}
From the two last statements we gather
\begin{align*}
&\sup_{r\ge 2R~\text{dyadic}}\max_{n=1,\cdots,N}(F_{n,r}\nabla(\phi,\sigma))^2\ll 1
\\&
\Longrightarrow
\frac{1}{R^2}\av_{(-R,R)^d}|(\phi,\sigma)-(\phi,\sigma)_{r_1}|^2dx\lesssim(\frac{r_1}{R})^2.
\end{align*}
In view of (\ref{o83}) this can be rewritten as
\begin{eqnarray}\label{o84}
\lefteqn{\forall\;r\ge 2R~\text{dyadic},\;n=1,\cdots,N\quad (F_{n,r}\nabla(\phi,\sigma))^2\ll 1}\nonumber\\
&\Longrightarrow&
\frac{1}{R^2}\av_{(-R,R)^d}|(\phi,\sigma)-(\phi,\sigma)_{r_1}|^2dx\le\frac{1}{2}(\frac{r_0}{R})^\beta f(\frac{R}{r_0}),
\end{eqnarray}
provided that we adjust the definition (\ref{o83}) of $r_1$ appropriately
(to obtain the estimate $\le\frac{1}{2}\cdot$ in (\ref{o84}) in place of just $\lesssim$).
We now turn to the second r.\ h.\ s.\ term in (\ref{o81}), which in view of the definition
of $(\phi,\sigma)_r$ we may estimate as follows
\begin{eqnarray*}
\lefteqn{\sum_{\stackit{r\in[2r_1,R]}{r\;\text{dyadic}}}\av_{(-R,R)^d}|(\phi,\sigma)_{\frac{r}{2}}-(\phi,\sigma)_{r}|^2dx}\\
&\le&\sum_{\stackit{r\in[2r_1,R]}{r\;\text{dyadic}}}\max_{Q\;\text{level}\;r}\max_{Q'\subset Q\;\text{level}\;\frac{r}{2}}
\big|\av_{Q'}(\phi,\sigma)-\av_{Q}(\phi,\sigma)\big|^2.
\end{eqnarray*}
Hence if for any of the $(\frac{R}{r})^d$ dyadic sub-cubes $Q$ of $(-R,R)^d$ of level $r$ 
we introduce the $N=2^d$ linear functionals $F_{Q,n}$ as an extension of
\begin{equation}\nonumber
F_{Q,n}\nabla\zeta:=\frac{1}{r}(\av_{Q'_n}\zeta dx-\av_{Q}\zeta dx),
\end{equation}
where $\{Q'_n\}_{n=1,\cdots,2^d}$ is an enumeration of the sub-cubes of level $\frac{r}{2}$ of $Q$, 
and which satisfy the desired boundedness property (\ref{o73}) restricted to gradient fields
(which is no issue because of Hahn-Banach extension) and translated (which will be no issue because
of stationarity), that is,
\begin{equation}\label{o86}
|F_{Q,n}\nabla\zeta|\lesssim\big(\av_{Q}|\nabla\zeta|^\frac{2d}{d+\beta}dx\big)^\frac{d+\beta}{2d},
\end{equation}
we have
\begin{eqnarray*}
\lefteqn{\frac{1}{R^2}\sum_{\stackit{r\in[2r_1,R]}{r\;\text{dyadic}}}\av_{(-R,R)^d}|(\phi,\sigma)_{\frac{r}{2}}-(\phi,\sigma)_{r}|^2dx}\nonumber\\
&\le&\sum_{\stackit{r\in[2r_1,R]}{r\;\text{dyadic}}}(\frac{r}{R})^2\max_{Q\;\text{level}\;r}\max_{n=1,\cdots,2^d}(F_{Q,n}\nabla(\phi,\sigma))^2.
\end{eqnarray*}
From this we learn, since for the auxiliary function $g(z):=\log^{-2}(z+e)$, 
the dyadic sum $\sum_{r\in[2r_1,R]}g(\frac{R}{r})$ is universally bounded,
\begin{eqnarray*}\nonumber
\lefteqn{\forall\;r\in[2r_1,R]\text{ dyadic},\;Q\;\mbox{level}\;r,\;n=1,\cdots,2^d}\nonumber\\
&\quad&
(F_{Q,n}\nabla(\phi,\sigma))^2\ll(\frac{R}{r})^{2}g(\frac{R}{r})(\frac{r_0}{R})^\beta f(\frac{R}{r_0})\nonumber\\
&\Longrightarrow&\frac{1}{R^2}\sum_{\stackit{r\in[2r_1,R]}{r\;\text{dyadic}}}\av_{(-R,R)^d}|(\phi,\sigma)_{\frac{r}{2}}-(\phi,\sigma)_{r}|^2dx
\le\frac{1}{2}(\frac{r_0}{R})^\beta f(\frac{R}{r_0}).
\end{eqnarray*}
In view of (\ref{o81}), the combination of this with (\ref{o84}) yields
\begin{eqnarray}\label{o85}
\lefteqn{\forall\;r\in[2r_1,R]\text{ dyadic},\;Q\;\mbox{level}\;r,\;n=1,\cdots,2^d}\nonumber\\
&\quad&
(F_{Q,n}\nabla(\phi,\sigma))^2\ll(\frac{R}{r})^{2}g(\frac{R}{r})(\frac{r_0}{R})^\beta f(\frac{R}{r_0})\nonumber\\
&\mbox{and}&\forall\;r\ge 2R~\text{dyadic},\;n=1,\cdots,N\quad (F_{n,r}\nabla(\phi,\sigma))^2\ll 1\nonumber\\
&\Longrightarrow&\frac{1}{R^2}\av_{(-R,R)^d}|(\phi,\sigma)-\av_{(-R,R)^d}(\phi,\sigma)|^2dx
\le(\frac{r_0}{R})^\beta f(\frac{R}{r_0}).
\end{eqnarray}

\medskip

Equipped with this deterministic argument, we now may proceed to the stochastic part: In 
the event of $r_*>r_0$, there exists a dyadic $R\ge r_0$ such that (\ref{o80}) holds, 
so that we learn from (\ref{o85}) that there exists 
\begin{itemize}
\item a (dyadic) $r\in[2r_1,R]$, a sub-cube $Q$ of $(-R,R)^d$ of level $r$, and an index $n=1,\cdots,2^d$ such
that $(F_{Q,n}\nabla(\phi,\sigma))^2\gtrsim(\frac{R}{r})^2g(\frac{R}{r})(\frac{r_0}{R})^\beta f(\frac{R}{r_0})$.
In view of the boundedness condition (\ref{o86}) and stationarity, we may apply (\ref{o72}) with $F$ replaced by 
$F_{Q,n}$ and $M^2$ replaced by $(\frac{R}{r})^2g(\frac{R}{r})(\frac{r_0}{R})^\beta f(\frac{R}{r_0})$. This $M$
is admissible in the sense of $M\lesssim 1$ because by (\ref{o83}) we have 
$(\frac{R}{r})^2g(\frac{R}{r})(\frac{r_0}{R})^\beta f(\frac{R}{r_0})
\le(\frac{R}{r_1})^2(\frac{r_0}{R})^\beta f(\frac{R}{r_0})\sim 1$.
Hence the probability of each single of this events is estimated as follows
\begin{eqnarray*}\nonumber
\lefteqn{\big\langle I\big((F_{Q,n}\nabla(\phi,\sigma))^2
\ge\frac{1}{C}(\frac{R}{r})^2g(\frac{R}{r})(\frac{r_0}{R})^\beta f(\frac{R}{r_0})\big)\big\rangle}\\
&\lesssim&\exp\big(-\frac{1}{C}(\frac{R}{r})^{2-\beta}g(\frac{R}{r}) f(\frac{R}{r_0}) r_0^\beta \big).
\end{eqnarray*}
Since $g(z)$ decays sub-algebraically in $z$ and since $\beta<2$, this yields the simpler form
\begin{eqnarray*}\nonumber
\lefteqn{\big\langle I\big((F_{Q,n}\nabla(\phi,\sigma))^2
\ge\frac{1}{C}(\frac{R}{r})^2g(\frac{R}{r})(\frac{r_0}{R})^\beta f(\frac{R}{r_0})\big)\big\rangle}\\
&\lesssim&\exp\big(-\frac{1}{C}(\frac{R}{r})^{1-\frac{\beta}{2}} f(\frac{R}{r_0}) r_0^\beta \big).
\end{eqnarray*}
\item {\it or} a (dyadic) $r\ge 2R$ and an index $n=1,\cdots,N$ for which the estimate $(F_{n,r}\nabla(\phi,\sigma))^2\gtrsim 1$ holds.
By the boundedness property of $F_{n,r}$, each single of these events is estimated as
\begin{equation}\nonumber
\big\langle I\big((F_{n,r}\nabla(\phi,\sigma))^2
\ge\frac{1}{C}\big)\big\rangle
\lesssim\exp(-\frac{1}{C}r^\beta).
\end{equation}
\end{itemize}
Taking the number $(\frac{R}{r})^d$ of sub-cubes $Q$ into account and recalling $N\lesssim 1$, this implies
\begin{align}
&\nonumber
\langle I(r_*>r_0)\rangle
\\
&\lesssim
\sum_{\stackit{R\ge r_0}{R\;\text{dyadic}}}
\Big(\sum_{\stackit{r\in[2r_0,R]}{r\;\text{dyadic}}}
(\frac{R}{r})^d\exp(-\frac{1}{C}(\frac{R}{r})^{1-\frac{\beta}{2}}f(\frac{R}{r_0})r_0^\beta)
+\sum_{\stackit{r\ge 2R}{r\;\text{dyadic}}}\exp(-\frac{1}{C}r^\beta)\Big).
\label{o88}
\end{align}
Again, since $1-\frac{\beta}{2}>0$, we have the calculus estimate
\begin{align*}\nonumber
&\sum_{\stackit{r\in[2r_0,R]}{r\;\text{dyadic}}}(\frac{R}{r})^d\exp\big(-A(\frac{R}{r})^{1-\frac{\beta}{2}}\big)
\\
&\lesssim
\exp(-A)\sum_{\stackit{r\in[2r_0,R]}{r\;\text{dyadic}}}(\frac{R}{r})^d\exp\big(-\frac{1}{C}A\log(\frac{R}{r})\big)\nonumber\\
&\lesssim
\exp(-A)\quad\mbox{for}\;A\gg 1.
\end{align*}
Applying this to the first sum over $r$ in (\ref{o88}) and $A=\frac{1}{C}f(\frac{R}{r_0})r_0^\beta$,
which satisfies $A\gg 1$ for $r_0\gg 1$, and using the estimate $\sum_{r\geq 2R; r\text{ dyadic}} \exp(-\frac{1}{C}r^\beta) \lesssim \exp(-\frac{1}{C} R^\beta)$ (which holds provided that $R\geq r_0\geq 1$) for the second sum over $r$, we obtain
\begin{eqnarray}
\langle I(r_*>r_0)\rangle
&\lesssim& \sum_{\stackit{R\ge r_0}{R\;\text{dyadic}}}\Big(\exp(-\frac{1}{C}f(\frac{R}{r_0})r_0^\beta)
+
\exp(-\frac{1}{C}R^\beta)\Big)
\quad\text{for }r_0\gg 1.\nonumber
\end{eqnarray}
Thanks to $\beta>0$, we have
$\exp(-\frac{1}{C}R^\beta)
\lesssim\exp(-\frac{1}{C}f(\frac{R}{r_0})r_0^\beta)$,
so that the second summand is dominated by the first one:
\begin{eqnarray}
\langle I(r_*>r_0)\rangle
&\lesssim& \sum_{\stackit{R\ge r_0}{R\;\text{dyadic}}}\exp(-\frac{1}{C}f(\frac{R}{r_0})r_0^\beta)
=\sum_{m=0}^\infty\exp(-\frac{1}{C}f(2^m)r_0^\beta)
.\nonumber
\end{eqnarray}
Now we see the reason for the choice of $f(z)=\log(e+\log z)$ for which $f(2^m)\ge\frac{1}{C}(1+\log(m+1))$
and thus 
\begin{equation}\nonumber
\sum_{m=0}^\infty\exp(-A f(2^m))\le\sum_{m=0}^\infty\exp(-\frac{A}{C}(1+\log(m+1)))\lesssim \exp(-\frac{A}{C})
\quad\mbox{for}\;A\gg 1.
\end{equation}
With $\frac{1}{C}r_0^\beta$ playing the role of $A$ this yields (\ref{o77}). Note that the condition $r_0\gg 1$ is immaterial after adjusting the constants, as the l.\ h.\ s.\ of (\ref{o77}) is bounded by $1$.
\end{proof}


\bibliographystyle{plain}
\bibliography{stochastic_homogenization}

\end{document}